\documentclass[twoside,11pt]{article}
\usepackage{tikz}
\usepackage{siunitx}
\usepackage{caption}
\usepackage{subcaption}
\usepackage[margin=0.25in]{geometry}
\usepackage{pgfplots}
\pgfplotsset{width=10cm,compat=1.9}
 \usepackage{multirow}
\usepackage{jmlr2e}
\usepackage{amsmath}
\usepackage[ruled,vlined]{algorithm2e}
\usepgfplotslibrary{external}
\firstpageno{1}

\begin{document}

\title{Sparse PCA: a Geometric Approach}

\author{\name Dimitris Bertsimas \email dbertsim@mit.edu \\
       \name Driss Lahlou Kitane \email driss@mit.edu \\
       \addr Operations Research Center\\
       Massachusetts Institute of Technology\\
       77, Massachusetts Ave.
       Cambridge, MA 02138, USA}

\editor{}

\maketitle

\begin{abstract}
We consider the problem of maximizing the variance explained from a data matrix using orthogonal sparse principal components that have a support of fixed cardinality. While most existing methods focus on building principal components (PCs) iteratively through deflation, we propose GeoSPCA, a novel algorithm to build all PCs at once while satisfying the orthogonality constraints which brings substantial benefits over deflation. This novel approach is based on the left eigenvalues of the covariance matrix which helps circumvent the non-convexity of the problem by approximating the optimal solution using a binary linear optimization problem that can find the optimal solution. The resulting approximation can be used to tackle different versions of the sparse PCA problem including the case in which the principal components share the same support or have disjoint supports and the Structured Sparse PCA problem. We also propose optimality bounds and illustrate the benefits of GeoSPCA in selected real world problems both in terms of explained variance, sparsity and tractability. Improvements vs. the greedy algorithm, which is often at par with state-of-the-art techniques, reaches up to 24\% in terms of variance while solving real world problems with 10,000s of variables and support cardinality of 100s in minutes. We also apply GeoSPCA in a face recognition problem yielding more than 10\% improvement vs. other PCA based technique such as structured sparse PCA.
\end{abstract}
\section{Introduction}

PCA \citep{pca} is a popular data analysis technique. It is used in a variety of applications including finance, data imputation, image processing and genome analysis. PCA is of particular interest when the data matrix data $\mathbf{X}\in\mathbb{R}^{n\times p}$ has a high dimension $p$. Yet, models resulting from PCA use all the features while sparsity of the PCs can be desired for various benefits. Sparse versions of PCA that use a reduced number of variables to build principal components were proposed to improve interpretability, enhance model's predictive power or reduce operational costs (such as in finance) and investment costs (such as in spectroscopy). Solving the sparse PCA problem is particularly challenging due to the non-convexity of the problem. 
\subsection{Background}
The optimization community has been studying several versions of sparse PCA problem for decades now. Most of the methods proposed fall into two broad categories. One of the categories aims at approximating the whole covariance matrix using sparse principal components; the loss function is typically of the form $||\mathbf{X} - \mathbf{Z}||$ where $\mathbf{X}\in\mathbb{R}^{n\times p}$ is the covariance matrix and $\mathbf{Z}$ is constructed using problem variables. In this category, sparsity is generated using thresholding and/or regularization either directly in the objective function or in the constraints. Thresholding  has been used as early as in \citep{Jeffers}. Limits of thresholding have also been documented \citep{cadima1995loading}, in particular key variables could be zeroed and highly correlated variables chosen together which ultimately could lead to inaccurate interpretations. The first algorithm to use $\ell_1$ penalization is SCOTLASS  \citep{scotlass}. It was introduced then as a method for preserving the orthogonality constraints while sparsity was induced by $\ell_1$ constraints but, using this technique, the number of features used in the resulting sparse model is capped by $n$ and the method is not tractable for $p\geq 100$. Larger problems could be then tackled by the introduction of the LASSO method to generate PCs by adopting an $\ell_1$ penalized regression approach in the generation of the PCs \citep{zou} which, on one hand, allowed the number of variables used in the model to exceed $n$ but, on the other hand, sacrificed the orthogonality of the PCs and provided no guarantee on the optimality of the solution. \citep{SHEN2008, witten} introduced iterative thresholding methods to tackle the sparse PCA problem building PCs iteratively. Other breakthroughs were achieved and successfully implemented including the GPower method using $\ell_1$ penalty which preserves orthogonality \citep{journee}, scales for $p$ in 10,000s and $n$ in 1,000s and outperforms LASSO-based regression approach and greedy algorithm in terms of quality of solution and computation speed. Subsequent works using penalization to induce sparsity adopting optimization over Stiefel manifolds \citep{huang,8807218,chenmani} and the Procrustes reformulation \citep{benidis} showed that solution quality and computation time could be further improved considering the same size of problems while preserving orthogonality. The state of the art in this category of works can then handle large instances in competitive time, delivers orthogonality of the PCs and can furthermore be adapted to variations of the sparse PCA problem including the structured sparse PCA \citep{bach,sspcabio}. Yet, the community is still making sizeable efforts to improve the variance captured. Indeed, most of these approaches come with no guarantee of the optimality of the solution and the control of the number of variables cannot be chosen with precision. Recent efforts in \citep{erichson} introduce $\ell_0$ penalization while keeping a loss function of the form $||\mathbf{X} - \mathbf{Z}||$ which could theoretically achieve sparser principal components and achieve a tighter control over the number of nonzero variables (numerical tests in this paper focusing essentially on $\ell_1$ penalization).
The number of variables used in the first category of works is often too high to enable interpretability or to be practical in several intended uses of sparse PCA. \medbreak 
The second category of approaches aims at having a strict control on the number of variables used. For a given number $k$ of variables, the objective sought is to maximize the variance captured using one or several principal components that have a support of cardinal $k$.  Rather than generating sparsity through penalization, this family of approaches aims at finding optimal or near optimal solutions for a defined number of features to be used. Initial work involved a greedy strategy as well as a branch-and-bound approach \citep{moghaddam}.  Optimal solutions were sought through semidefinite optimization modeling (SDO) \citep{aspremont1}. Further developments of this technique enabled tackling relatively big problems with optimality certificates \citep{aspremont2}. GPower technique was also adapted to propose a truncation approach in \citep{gpowert}. Later, a novel technique that performs particularly well when the the decay of eigenvalue of the data matrix $\mathbf{X}$ was proposed \citep{papailiopoulos13}. Although each sparse PC constructed is optimal, the set of PCs are constructed iteratively using deflation techniques of the data matrix, thus there is no guarantee of the overall optimality and orthogonality is sacrificed. Another track of work develops branch-and-bound approaches to construct optimal solutions while controlling $k$. This enabled solving problems for $p$ in the 1,000s and $k\leq10$ but the PCs are then constructed by iteratively deflating $\mathbf{X}$ \citep{lauren}. More recent work combines branch-and-bound and SDO formulation to solve problems with $p$ in 10,000s and $k\leq 10$ to near optimality in hours \citep{ryan, li2020exact}. Since the PCs are obtained by deflating $\mathbf{X}$, orthogonality is sacrificed. One variation of the sparse PCA tackled by this category of approaches includes additional constraints aiming at building sparse PCs with disjoint supports \citep{texas}. More recently, \citep{delpia} proves that finding orthogonal principal components sharing the same support and maximizing the variance captured is a polynomial problem when the rank of the covariance matrix is fixed. The result is further extended to a special case of the disjoint supports sparse PCA problem when the cardinality of the disjoint supports is identical. This paper is theoretical and does not provide any numerical results so there is no indication regarding how the suggested approach would perform from a practical point of view. \citep{dey} also recently proposed the first algorithm that builds simultaneously orthogonal sparse principal components sharing the same support while upper-bounding the number of nonzero rows. The methods tackles problems for $p$ as high as 2000 in less than two hours and upper and lower bounds are provided for the optimal solution.\medbreak
The two broad categories of approaches presented do not span all of the approaches that the machine learning community designed. Indeed, remarkable approaches aimed for instance at approximating the subspace generated by the principal components of the classic PCA algorithm were developed. \citep{Johnstone} focuses on the principal components that correspond to the largest eigenvalues of the covariance matrix and generates sparsity through thresholding. \citep{ma} further improves the results achieved by this approach by proposing a novel iterative thresholding technique achieving tighter loss than comparable techniques.
\subsection{Motivations and Contributions}
As mentioned in the background section, one of the main approaches to construct sparse principal components that the machine learning community adopted is to add a constraint upper-bounding the number of nonzero variables. This approach is particularly relevant when the desired number of nonzero variables is low compared to the dimension $p$ of the data points. When upper-bounding the number of nonzero variables by an integer $k$ when $k$ is significantly lower, the maximization of the variance captured is more relevant than the minimization of the error $||\mathbf{X} - \mathbf{Z}||$ where $\mathbf{X}\in\mathbb{R}^{n\times p}$ is the data matrix and $Z$ is the constructed sparse matrix as the error $||\mathbf{X} - \mathbf{Z}||$ is too large. Considering this approach, the machine learning community focused on constructing the PCs iteratively which does not guarantee the optimality of the overall solution (and actually yields sub-optimal solutions). Only very recent works \citep{delpia, dey} tackle the problem of building multiple orthogonal principal components at once to guarantee optimality when PCs share a common support \citep{delpia, dey} or when PCs have disjoint supports \citep{delpia}. Our main motivation is to bring a new method that upper-bounds the number of nonzero variables, builds multiple orthogonal principal components at once and scales further than existing methods when PCs share the same support or have block-disjoint supports. Furthermore, we also aim at providing guarantees on the quality of the solution found. \medbreak
In the present paper, we propose a novel approach to sparse PCA that  considers the left eigenvectors of $\mathbf{X}$ (it is worth noting that all works on sparse PCA focus on the right eigenvectors). We derive then a geometrical interpretation of the resulting problem. This interpretation leads to a binary linear formulation that aims at approximating the original problem. This geometrical approach is versatile and can be adapted to several versions of sparse PCA problem. We introduce two formulations; one for a version in which all PCs share the same support and another version in which groups of PCs use disjoint supports (a generalization of the version imposing to PCs to have disjoint supports). Building on the versatility of the method, we also propose a formulation for the structured sparse PCA problem. We provide optimality gap bounds and test the proposed method on real world data sets. The geometrical approach we propose solves problem for $p$ in 10,000s and $k$ in 100s  while preserving the orthogonality and controlling $k$ in minutes. \medbreak
As we mentioned earlier, upper-bounding the number of nonzero variables is a desirable feature for practitioners whether the \textit{true} principal components share the same support or not. While the method we propose could be used in the general case, practitioners might me even more interested in this technique when the data studied follows particular structures (eg. hyperspectral imagery and spectroscopy when the focus is on particular variables \citep{spectro2, spectro1}, computer vision \citep{bach} or matching \citep{benidis} among others).
\medbreak
We summarize our contributions in this paper below:
\begin{itemize}
    \item We introduce a new approach to the sparse PCA problem based on left eigenvectors that leads to a geometrical interpretation of sparse PCA. We approximate the sparse PCA problem using binary linear optimization (BLO) and the introduction of cuts that improve the solution.
    \item We prove that the optimal solution can be found using the approximation proposed in a finite number of steps and provide a theoretical optimality gap to the solution generated by the method we propose.
    \item We propose formulations and algorithms (i) for sparse PCA problem in which all PCs share the same support, (ii) for a generalization of the version of PCA requiring that the PCs have disjoint support, and (iii) for the structured sparse PCA problem.
    \item We test the proposed method on publicly available real world data sets and compare its performance vs. existing methods and show that the geometrical approach produces solutions of higher quality than alternative state-of-the-art methods.
\end{itemize}
The structure of the paper is as follows. In Section 2, we first study the sparse PCA problem when all principal components share the same support. We introduce the geometrical approach and offer interpretations and provide formulations, optimality gap bounds and related algorithm. In Section 3, we extend the geometric approach to the case in which the principal components have disjoint supports and to the structured sparse PCA problem. In Section 4, we compare GeoSPCA to other state-of-the-art sparse PCA techniques and we finally conclude in Section 5.\medbreak
\subsection{Notations and Definitions}
In the remainder of this paper, we define $\mathbf{X}\in\mathbb{R}^{n\times p}$ a centered data matrix with $n$ data points and $p$ features. We refer to the number of features used to build a sparse PCA model by $k$. All matrices and vectors in bold characters. \medbreak 
We note $\mathbf{M}^T$, the transpose of matrix $\mathbf{M}$. Consider the SVD decomposition of a real matrix $\mathbb{R}^{n\times p}$, $\mathbf{U}\mathbf{\Sigma}\mathbf{V}^T$. We call the columns of $\mathbf{U}$ (resp. $\mathbf{V}$) the left (resp. right) eigenvectors of $\mathbf{X}$. Consider $m$ a non-negative integer, we note $[m]=\{1,2,\dots,m\}$. We denote $\mathbf{X}_i$ the $i^{th}$ column of a matrix $\mathbf{X}$. The scalar $s_i$ is the $i^{th}$ component of vector $\boldsymbol s$. The matrix $\mathbf{S} = diag(\mathbf{s})$ is the diagonal matrix with a diagonal equal to the vector $\boldsymbol s$. Consider $\sigma\subset [m]$, we note $\boldsymbol s^\sigma$ the vector of $\{0,1\}^m$ such that $s^\sigma_i = 1$ if $i\in\sigma$ and $s^\sigma_i = 0$, otherwise. If $E$ is a finite set, we note $|E|$ its cardinality. Consider a matrix $\mathbf{M}\in\mathbb{R}^{n\times p}$ and a vector $\boldsymbol s \in \{0,1\}^p$, we note $\mathbf{M}(\mathbf{s})$ the matrix we obtain by suppressing all columns $\mathbf{M}_i$ of $\mathbf{M}$ such that $s_i=0$. $||.||_F$ is the Frobenius norm. $||.||_0$ designates the number of nonzero coefficients of a vector or the number of nonzero rows of a matrix. For $b$ a non-negative integer, $\mathbf{I}_b$ is the identity matrix of size $b \times b$. Given $\boldsymbol s\in \{0,1\}^p$, we note $\mathbf{U}[\mathbf{s}]$ any element of $argmax_{\mathbf{U}\in\mathbb{R}^{a\times n}}\,tr(\mathbf{U}^T\mathbf{X}(\mathbf{s})\mathbf{X}(\mathbf{s})^T\mathbf{U})\,s.t.\,\mathbf{U}^T\mathbf{U}=\mathbf{I}_a$ for a given  non-negative integer $a$ and $\eta(\mathbf{s}) = ||\mathbf{X}(\mathbf{s}) - \mathbf{U}[\mathbf{s}]\mathbf{U}[\mathbf{s}]^T\mathbf{X}(\mathbf{s})||^2_F$, $\mu(\mathbf{s}) = ||\mathbf{X}(\mathbf{s})||_F^2 = \sum_{i=1}^p s_i ||\mathbf{X}_i||^2$ and $\pi(\mathbf{s}) = ||\mathbf{U}[\mathbf{s}]\mathbf{U}[\mathbf{s}]^T\mathbf{X}(\mathbf{s})||^2_F = tr(\mathbf{U}[\mathbf{s}]^T\mathbf{X}(\mathbf{s})\mathbf{X}(\mathbf{s})^T\mathbf{U}[\mathbf{s}])$. It is easy to verify that $\mu(\mathbf{s}) = \pi(\mathbf{s}) + \eta(\mathbf{s})$ (See Figure \ref{fig0}). We define finally the vector $\mathbf{e}\in \mathbb{R}^p$ as the vector $\mathbf{e} = \{1,1,\dots,1\}$.
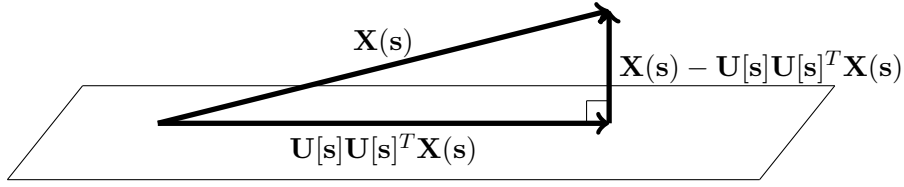
\begin{figure}
\centering
\begin{tikzpicture}
    \draw   (5,-0.25) -- (-5, -0.25);
    \draw   (6,1) -- (-4, 1);
    \draw   (6,1) -- (5, -0.25);
    \draw   (-4,1) -- (-5, -0.25);
    \draw (3,0.8) -- (2.7,0.8);
    \draw (2.7,0.5) -- (2.7,0.8);
    \draw[line width=2pt,black,-stealth][->] (-3,0.5)--(3,2) ;
    \node[align=left, above] at (0,1.25)%
    {$\mathbf{X}(\mathbf{s})$};
    \draw[line width=2pt,black,-stealth][->] (-3,0.5)--(3,0.5);
    \node[below] at (0,0.5) {$\mathbf{U}[\mathbf{s}]\mathbf{U}[\mathbf{s}]^T\mathbf{X}(\mathbf{s})$};
  \draw[line width=2pt,black,-stealth][->] (3,0.5)--(3,2);
    \node[right] at (3,1.25) {$\mathbf{X}(\mathbf{s}) - \mathbf{U}[\mathbf{s}]\mathbf{U}[\mathbf{s}]^T\mathbf{X}(\mathbf{s})$};
\end{tikzpicture}
\caption{$\mathbf{X}(\mathbf{s}) - \mathbf{U}[\mathbf{s}]\mathbf{U}[\mathbf{s}]^T\mathbf{X}(\mathbf{s})$ is the projection of $\mathbf{X}(\mathbf{s})$ on the subspace of lower dimension defined by $\mathbf{U}[\mathbf{s}]$} \label{fig0}
\end{figure}

\section{Sparse PCA problem with Principal components sharing the same support}
In this section, we address the sparse PCA problem in which all the principal components share the same support. We first transform the problem to a left eigenvectors problem and derive a geometric interpretation of the problem. We then introduce a binary linear optimization approximation of the sparse PCA problem. We finally propose an algorithm that solves the approximation and that can find the optimal solution to the original problem. 
\subsection{An exact formulation using the left eigenvectors perspective}
We consider the problem of maximizing the variance explained by a given number $a$ of sparse principal components that share the same support of cardinal less than or equal to $k$. We write the problem as:

\begin{align*}
\tag{1}
    \displaystyle \max_{\mathbf{W}\in \mathbb{R}^{p\times a}} & \,tr(\mathbf{W}^T\mathbf{X}^T\mathbf{X}\mathbf{W})\\
    s.t.&\,||\mathbf{W}||_0 \leq k,\\
    & \mathbf{W}^T\mathbf{W} = \mathbf{I}_a,
\end{align*}
where $||\mathbf{W}||_0$ is the number of nonzero rows of $\mathbf{W}$. The $a$ principal components formed by the columns of $\mathbf{W}$ share then the same support of cardinality at most $k$.  We transform the problem to let the left eigenvectors appear: 
\begin{proposition}
Problem (1) is equivalent to:
\begin{align*}
\tag{2}
   \displaystyle \max_{\boldsymbol s\in \{0,1\}^p,\boldsymbol U \in \mathbb{R}^{n\times a}}\, &\, \sum_{i=1}^ps_i\sum_{j = 1}^a(\mathbf{X}_i^T\mathbf{U}_j)^2\\
    s.t.&\,\mathbf{e}^T\mathbf{s} \leq k, \\
    & \mathbf{U}^T\mathbf{U} = \mathbf{I}_a.\\
\end{align*}
\end{proposition}
\begin{proof}
We introduce a variable $\boldsymbol s \in\{0,1\}^p$ and consider that $s_i = 0$ if $\mathbf{W}^T_i = \mathbf{0}$ and $s_i = 1$, otherwise. When $\sum_{i=1}^p s_i\leq k$, and $\mathbf{W}\in \mathbb{R}^{p\times a}$, we have $||\mathbf{S}\mathbf{W}||_0 \leq k$ where $\mathbf{S} = diag(\mathbf{s})$ and $||\mathbf{S}\mathbf{W}||_0$ is the number of nonzero rows of $\mathbf{SW}$. If the principal components are represented by the columns of $\mathbf{SW}$, then $\mathbf{W}^T\mathbf{S}^T\mathbf{S}\mathbf{W} = \mathbf{W}^T\mathbf{S}\mathbf{W} = \mathbf{I}_a$ ensures that the principal components are normalized and orthogonal. We note that if, in addition, we have $\mathbf{W}^T\mathbf{W} = \mathbf{I}_a$, we have $\forall i \in [p]$, $\sum_{j=1}^p W_{ij}^2 = 1$ but since $\sum_{j=1}^p s_jW_{ij}^2 = 1$, it means that $W_{ij} = 0$ if $s_j =0$. Hence, the combination of $\mathbf{W}^T\mathbf{S}\mathbf{W} = \mathbf{I}_a$ and $\mathbf{W}^T\mathbf{W}=\mathbf{I}_a$ ensures that $||\mathbf{W}||_0 \leq k$. Problem (1) can then be rewritten:

\begin{align*}
\tag{3}
    \displaystyle\max_{\boldsymbol s\in \{0,1\}^p,\mathbf{W}\in \mathbb{R}^{p\times a}} \,& tr(\mathbf{W}^T\mathbf{S}\mathbf{X}^T\mathbf{X}\mathbf{S}\mathbf{W})\\
    s.t.\,&\mathbf{e}^T\mathbf{s} \leq k,\\
    & \mathbf{W}^T\mathbf{S}\mathbf{W}= \mathbf{I}_a,\\
    & \mathbf{W}^T\mathbf{W} = \mathbf{I}_a.
\end{align*}
\medbreak

The constraints $\mathbf{W}^T\mathbf{S}\mathbf{W}= \mathbf{I}_a$ can be dropped (the proof is detailed in Proposition 2). Consider $\mathbf{X}(\mathbf{s})$ the matrix obtained when the columns $\mathbf{X}_i$ are removed from $\mathbf{X}$ when $s_i = 0$. Problem (3) can then be written considering inner and outer problems as follows:

\begin{align*}
     \displaystyle\max_{\boldsymbol s\in \{0,1\}^p, \mathbf{W}\in \mathbb{R}^{k\times a}} \,&   tr(\mathbf{W}^T\mathbf{X}(\mathbf{s})^T\mathbf{X}(\mathbf{s})\mathbf{W})\\
    s.t.\,&\mathbf{e}^T\mathbf{s} \leq k \\
    &  \mathbf{W}^T\mathbf{W} = \mathbf{I}_a.\\
\end{align*}

The inner problem is a standard PCA problem for the data matrix $\mathbf{X}(\mathbf{s})$. We can now write the equivalent problem considering the left eigenvectors of $\mathbf{X}(\mathbf{s})$:

\begin{align*}
\begin{aligned}
    \displaystyle\max_{\boldsymbol s\in \{0,1\}^p,\boldsymbol U \in \mathbb{R}^{n\times a}} \,& tr(\mathbf{U}^T\mathbf{X}(\mathbf{s})\mathbf{X}(\mathbf{s})^T\mathbf{U})\\
    s.t.\,&\mathbf{e}^T\mathbf{s} \leq k\, \\
    &\mathbf{U}^T\mathbf{U} = \mathbf{I}_a.\\
\end{aligned}
\end{align*}
By expanding the trace, we have then:
\begin{align*}
\begin{aligned}
    \displaystyle\max_{\boldsymbol s\in \{0,1\}^p,\boldsymbol U \in \mathbb{R}^{n\times a}} \,&  \sum_{i=1}^k\sum_{j = 1}^a(\mathbf{X}(\mathbf{s})_i^T\mathbf{U}_j)^2\\
    s.t.\,&\mathbf{e}^T\mathbf{s} \leq k,\\
    & \mathbf{U}^T\mathbf{U} = \mathbf{I}_a,\\
\end{aligned}
\end{align*}
which is equivalent to:
\begin{align*}
\begin{aligned}
    \displaystyle\max_{\boldsymbol s\in \{0,1\}^p,\boldsymbol U \in \mathbb{R}^{n\times a}} \,&\sum_{i=1}^ps_i\sum_{j = 1}^a(\mathbf{X}_i^T\mathbf{U}_j)^2\\
    s.t.\,&\mathbf{e}^T\mathbf{s} \leq k\\
    &\mathbf{U}^T\mathbf{U} = \mathbf{I}_a,\quad \quad \quad\\
\end{aligned}
\end{align*}
as $ \displaystyle\sum_{i=1}^k\sum_{j = 1}^a(\mathbf{X}(\mathbf{s})_i^T\mathbf{U}_j)^2 = \sum_{i=1}^ps_i\sum_{j = 1}^a(\mathbf{X}_i^T\mathbf{U}_j)^2$.
\end{proof}
$\sum_{j = 1}^a(\mathbf{X}_i^T\mathbf{U}_j)^2$ is the sum of the norms of the projections of $\mathbf{X}_i$ in the subspace generated by the columns of $\mathbf{U}$. In other words, the problem is to find $k$ columns of $\mathbf{X}$ that maximize the sum of the norms of their projections in a subspace of $\mathbb{R}^{n}$ of dimension $a$.\medbreak

\begin{proposition}
Problem $(3)$ is equivalent to:
\begin{align*}
    \max_{\boldsymbol s\in \{0,1\}^p,\mathbf{W}\in \mathbb{R}^{p\times a}} & \,tr(\mathbf{W}^T\mathbf{S}\mathbf{X}^T\mathbf{X}\mathbf{S}\mathbf{W})\\
    s.t.&\,\mathbf{e}^T\mathbf{s} \leq k,\\
    & \mathbf{W}^T\mathbf{W} = \mathbf{I}_a.
\end{align*}
\end{proposition}
\begin{proof}
Consider an SVD decomposition of $\mathbf{X}(\mathbf{s}) = \mathbf{U\boldsymbol\Sigma V}^T$ with $\mathbf{U} \in \mathbb{R}^{n\times n}$, $\boldsymbol\Sigma \in \mathbb{R}^{n \times k}$ and $\mathbf{V}\in \mathbb{R}^{k \times k}$ such that $\boldsymbol\Sigma$ is a diagonal matrix and $\mathbf{U}^T\mathbf{U} = \mathbf{V}^T\mathbf{V} = \mathbf{I}_k$. Consider now $\overline{\mathbf{V}} \in \mathbb{R}^{p \times k}$ such that $\overline{\mathbf{V}}_i = 0$ if $s_i = 0$ and the columns of $\mathbf{V}$ appear in $\overline{\mathbf{V}}$ in the indices $i$ for which $s_i = 1$ in the same order as in $\mathbf{V}$. We can easily verify that $\mathbf{XS} = \mathbf{U\boldsymbol\Sigma V}'^T$ which shows that $\mathbf{XS}$ and $\mathbf{X}(\mathbf{s})$ have the same nonzero eigenvalues and the same left eigenvectors.
\end{proof}

\subsection{Geometric approximation formulation}

Although the problem has been transformed, the two main issues of the problem are still present; the objective function is not concave and we still need to deal with the non-convex orthogonality constraints. In this subsection, our objective is to provide a binary linear problem that approximates Problem (2). We first propose a linear approximation of the objective function and outline the geometric intuition supporting this approximation. We then replace the non-convex constraints by a set of linear constraints yielding the same optimal solutions. The approximation introduces a new parameter (noted $\eta$). We also show that optimal solutions of Problem (2) are also optimal solutions for the approximation if $\eta$ is chosen appropriately. 

\medbreak We start by addressing the non-concavity of the objective function. The underlying hypothesis when using PCA to tackle a data problem is that the data matrix $\mathbf{X}$ can be written $\mathbf{X} = \mathbf{X}' + \boldsymbol\epsilon$ where $\mathbf{X}'$ is a matrix of rank $a<min(n,p)$ and $\boldsymbol\epsilon$ is a low-norm matrix representing noise, or second order phenomena or other perturbations that are to be ignored by PCA modelling. $\mathbf{X}'$ is found by projecting $\mathbf{X}$ according to a set of $a$ orthonormal vectors $\mathbf{V}_i$, the columns of $\mathbf{V}$, $\mathbf{X}'=\mathbf{VV}^T\mathbf{X}$. Since $||\mathbf{X} - \mathbf{X}'||_F = ||\boldsymbol\epsilon||_F \approx 0$, we consider that $||\mathbf{X}||_F\approx ||\mathbf{X}'||_F$.\medbreak
When considering Problem $(2)$, the objective function is the sum of the norms of the projections of $k$ columns of $\mathbf{X}$ on the subspace generated by the columns of $\mathbf{U}$. Since the inner problem of problem $(2)$ is the standard PCA problem, and using the approximation we just mentioned, we can approximate the objective function as the sum of the norms of $k$ columns of $\mathbf{X}$ instead of the sum of the norms of their projections whenever $||\mathbf{X} - \mathbf{UU}^T\mathbf{X}||_F$ is small enough. We introduce then the parameter $\eta$ that bounds that norm of the difference of $\mathbf{X}$ and its projection $\mathbf{UU}^T\mathbf{X}$. This can be expressed by a constraint $||\mathbf{X} - \mathbf{UU}^T\mathbf{X}||_F^2\leq \eta$ for a given $\eta \geq 0$:
\begin{align*}
   \displaystyle \max_{\boldsymbol s\in \{0,1\}^p,\boldsymbol U \in \mathbb{R}^{n\times a}}\, &\, \,\sum_{i=1}^p s_i||\mathbf{X}_i||^2\\
    s.t.&\,\mathbf{e}^T\mathbf{s} \leq k, \\
    & \mathbf{U}^T\mathbf{U} = \mathbf{I}_a,\\
    &\,||\mathbf{X}(\mathbf{s}) - \mathbf{UU}^T\mathbf{X}(\mathbf{s})||_F^2 \leq \eta \\
\end{align*}
that could be rewritten as follows:
\begin{align*}
\tag{5}
\begin{aligned}
    \max_{\boldsymbol s\in \{0,1\}^p, \mathbf{e}^T\mathbf{s} \leq k}  \,&\, \max_{\boldsymbol U \in \mathbb{R}^{n\times a}}   & \,\sum_{i=1}^ps_i||\mathbf{X}_i||^2 \quad\quad\quad\quad\quad\quad\quad\\
     &\quad  s.t. &\, \mathbf{U}^T\mathbf{U} = \mathbf{I}_a,\quad\quad\quad\quad\quad\quad\quad\\
     & &\,||\mathbf{X}(\mathbf{s}) - \mathbf{UU}^T\mathbf{X}(\mathbf{s})||_F^2 \leq \eta. \\
\end{aligned}
\end{align*}
 
We now consider the constraints $\mathbf{U}^T\mathbf{U} = \mathbf{I}_a$. These constraints are difficult to work with as they are non-convex. Yet, $\mathbf{U}$ does not appear in the objective function considered now. For a given $\mathbf{s} \in\{0,1\}^p$, if we can find a sufficient condition for the existence of a $\mathbf{U}$ that satisfies $||\mathbf{X}(\mathbf{s}) - \mathbf{UU}^T\mathbf{X}(\mathbf{s})||_F^2 \leq \eta$, then $s$ is feasible. Consider now $\mathbf{U}[\mathbf{s}]\in argmin_\mathbf{U}\,||\mathbf{X}(\mathbf{s}) - \mathbf{UU}^T\mathbf{X}(\mathbf{s})||_F^2 \,s.t.\,\mathbf{U}^T\mathbf{U} = \mathbf{I}_a$. $\mathbf{U}[\mathbf{s}]$ is a solution of the standard PCA problem for matrix $\mathbf{X}(\mathbf{s})$. 
\begin{proposition}
Problem $(5)$ is equivalent to the following formulation $(6)$ for any $\mathbf{U}[\mathbf{s}]\in argmin_\mathbf{U}\,||\mathbf{X}(\mathbf{s}) - \mathbf{UU}^T\mathbf{X}(\mathbf{s})||_F^2\,s.t.\,\mathbf{U}^T\mathbf{U} = \mathbf{I}_a$:
\begin{align*}
\tag{6}
    \max_{\boldsymbol s\in \{0,1\}^p} & \,\sum_{i=1}^ps_i||\mathbf{X}_i||^2&\\
    s.t.\,&\mathbf{e}^T\mathbf{s} \leq k, \\
    & \eta(\mathbf{s}) \leq \eta.\\
\end{align*}
\end{proposition}

\begin{proof}
The objective function of the inner problem of $(5)$ depends only on the existence of $\mathbf{U}$ such that $||\mathbf{X}(\mathbf{s}) - \mathbf{UU}^T\mathbf{X}(\mathbf{s})||_F^2 \leq \eta$ and $\mathbf{U}^T\mathbf{U}=\mathbf{I}_a$. If such $\mathbf{U}$ exists, the objective function is equal to $\sum_{i=1}^ps_i||\mathbf{X}_i||^2$. If $\mathbf{U}[\mathbf{s}]$ verifies $||\mathbf{X}(\mathbf{s}) - \mathbf{U}[\mathbf{s}]\mathbf{U}[\mathbf{s}]^T\mathbf{X}(\mathbf{s})||_F^2 \leq \eta$, then $\mathbf{U}[\mathbf{s}]$ is a feasible solution of the inner problem of $(5)$. Conversely, if the inner problem of $(5)$ is feasible then by definition $||\mathbf{X}(\mathbf{s}) - \mathbf{U}[\mathbf{s}]\mathbf{U}[\mathbf{s}]^T\mathbf{X}(\mathbf{s})||_F^2 \leq \eta$.
\end{proof}
We replace finally the constraints $\eta(\mathbf{s}) \leq \eta$ by eliminating the vectors $\boldsymbol s$ for which we have  $\eta(\mathbf{s}) > \eta$. Consider $\sigma \subset [p]$, and $\mathbf{s}^{\sigma}\in\{0,1\}^p$ such that $s^{\sigma}_i=1$ if $i\in \sigma$ and $s^{\sigma}_i=0$ otherwise. If $\eta(\mathbf{s^\sigma}) > \eta$, $\mathbf{s}^{\sigma}$ is not feasible and could be cut using the following inequality:
$$\sum_{i\in\sigma}s_i \leq |\sigma| - 1.$$
Hence, we derive the following approximate formulation for the sparse PCA problem with PCs sharing a common support:
\begin{align*}
\tag{7}
    \max_{\boldsymbol s\in \{0,1\}^p} & \,\sum_{i=1}^ps_i||\mathbf{X}_i||^2\\
     s.t.\,&\mathbf{e}^T\mathbf{s} \leq k,\\
    & \forall \sigma \subset [p], \eta(\mathbf{s}^{\sigma}) > \eta\\
    &\Rightarrow \sum_{i\in\sigma}s_i \leq |\sigma| - 1.  \\
\end{align*}
We note $f(\eta)$ the value of the objective function of problem $(2)$ evaluated for an arbitrarily chosen optimal solution of problem $(7)$ for a chosen $\eta >0$ (if problem $(7)$ has several solutions, we choose one that minimizes $||\mathbf{X}(\mathbf{s})-\mathbf{U}^T\mathbf{U}\mathbf{X}(\mathbf{s})||_F$) and let us prove that an optimal solution to problem $(2)$ can be found using formulation (7):
\begin{theorem}
Consider $\mathbf{s}^o$ an optimal solution of problem $(2)$. There exists $\delta > 0$ such that for any $\eta \in [\eta(\mathbf{s}^o),\eta(\mathbf{s}^o) + \delta]$, optimal solutions of problem $(7)$ are optimal solutions of problem $(2)$.
\end{theorem}
\begin{proof}
By contradiction, consider $\eta = \eta(\mathbf{s}^o)$ and suppose that there exists $\mathbf{s}'\in\{0,1\}^p$ such that $\mathbf{s}'$ is an optimal solution of problem $(7)$ and is not an optimal solution for problem $(2)$ which implies $\sum_{i=1}^ps^o_i\sum_{j=1}^a(\mathbf{U}_j[\mathbf{s}^o]^T\mathbf{X}_i)^2 > \sum_{i=1}^ps'_i\sum_{j=1}^a(\mathbf{U}_j[\mathbf{s}']^T\mathbf{X}_i)^2$ according to problem $(2)$. Since $\mathbf{s}'$ is an optimal solution of problem $(7)$ and $s^o$ is feasible in (7) because $\eta = \eta(\mathbf{s}^o)$, we have $\sum_{i = 1}^ps'_i||\mathbf{X}_i||^2 \geq \sum_{i = 1}^ps^o_i||\mathbf{X}_i||^2$, which yields then:

$$\sum_{i = 1}^ps'_i||\mathbf{X}_i||^2 - \sum_{i=1}^ps'_i\sum_{j=1}^a(\mathbf{U}_j[\mathbf{s}']^T\mathbf{X}_i)^2 > \sum_{i = 1}^ps^o_i||\mathbf{X}_i||^2 - \sum_{i=1}^ps^o_i\sum_{j=1}^a(\mathbf{U}_j[\mathbf{s}^o]^T\mathbf{X}_i)^2, $$
i.e.
$$\eta(\mathbf{s}') > \eta(\mathbf{s}^o), $$
which means that $\mathbf{s}'$ is not feasible as it violates the constraint $\eta(\mathbf{s}) \leq \eta(\mathbf{s}^o)$.\medbreak

Since $f(\eta)$ takes discrete and finite values as the number of $\boldsymbol s \in\{0,1\}^p$ is finite, we can choose $\delta>0$ such that the set of feasible solutions of problem $(7)$ remains the same as if $\eta = \eta(\mathbf{s}^o)$ to ensure that $\mathbf{s}^o$ remains an optimal solution of problem $(7)$. For example, we can consider $\mathbf{\Tilde{s}}\in argmin_\mathbf{s}||\mathbf{X}(\mathbf{s}) - \mathbf{U}[\mathbf{s}]\mathbf{U}[\mathbf{s}]^T\mathbf{X}(\mathbf{s})||_F^2 > \eta(\mathbf{s}^o)$ and choose $\delta = (\Tilde{\eta} - \eta(\mathbf{s}^o))/2$ with $\Tilde{\eta} = ||\mathbf{X}(\mathbf{\Tilde{s}}) - \mathbf{U}[\mathbf{\Tilde{s}}]\mathbf{U}[\mathbf{\Tilde{s}}]^T\mathbf{X}(\mathbf{\Tilde{s}})||_F^2$.
\end{proof}
Beyond the theoretical value of Theorem 4 as it shows that with an adequate $\eta$, an optimal solution could be found. More practically, this theorem is also used in the design of a practical algorithm in Section 2.4 (Algorithm 2) can approach and find the optimal solution in a finite number of steps.\medbreak
Although the objective function and the constraints of problem $(7)$ are linear, there are still potentially an exponential number of constraints (the number of subsets of [p] is $2^p$) which can hinder the tractability of the problem and we still need to assess the quality of the solution found by solving problem $(7)$ relative to the optimal solution of $(1)$. We start by providing a theoretical bound that partially addresses the first question and further address practical and empirical aspects in Section 3. We also address the large number of constraints by proposing a simple separation procedure to generate useful constraints without including all the constraints in the problem.\medbreak
Finally we note that the approximation $||\mathbf{\epsilon}||_F\approx 0$ is considered only to provide the reader with an intuition behind the approximation. $||\mathbf{\epsilon}||_F\approx 0$ is not a required assumption in any of the results of this paper.
\subsection{Worst case upper bound}
Before introducing an algorithm to solve Problem (7), we provide tight worst case upper bounds for the difference between the optimal value of Problem (2) and the value of the objective function of Problem (2) evaluated at an optimal solution of the approximation Problem (7). After showing the role $\eta$ plays in these bounds, we derive a straightforward way to derive an upper bound without knowing $\eta$. The relationship between the worst case upper bound and $\eta$ furthermore enables us to design an algorithm that approaches an optimal solution of the original Problem (2). Furthermore, this algorithm finds the optimal solution to the original Problem (2) in a finite number of iterations.
\begin{proposition}
Consider $(\mathbf{s}^o,\mathbf{U}[\mathbf{s}^o])$ an optimal solution of problem $(2)$.   Let $\mathbf{s}^A$ be an optimal solution of problem $(7)$ for $\eta \geq \eta(\mathbf{s}^o)$, then $\pi(\mathbf{s}^o) - \pi(\mathbf{s}^A) \leq \eta$; meaning that the difference between the optimal value of problem $(2)$,$ \sum_{i=1}^ps^o_i\sum_{j = 1}^a(\mathbf{X}_i^T\mathbf{U}[\mathbf{s}^o]_j)^2$, and the value of the objective function of problem $(2)$ evaluated at $\mathbf{s}^A$, $ \sum_{i=1}^ps^A_i\sum_{j = 1}^a(\mathbf{X}_i^T\mathbf{U}[\mathbf{s}^A]_j)^2$ is bounded by $\eta$. Furthermore, this bound is tight.
\end{proposition}

\begin{proof}
Note $\mathbf{V}^o$ the optimal value of problem $(2)$, $\mathbf{s}^A$ a solution to problem $(7)$ and $\mathbf{V}^A = \max_{\boldsymbol U \in\mathbb{R}^{n\times a}}$ $\sum_{i=1}^ps^A_i\sum_{j = 1}^a(\mathbf{X}_i^T\mathbf{U}_j)^2$ s.t. $\mathbf{U}^T\mathbf{U}=\mathbf{I}_a$, the value of the objective function of problem $(2)$ evaluated at $\mathbf{s}^A$. We have: 
$$\sum_{i=1}^p s^o_i||\mathbf{X}_i||^2 \geq \mathbf{V}^o$$
and $$ \mathbf{V}^A \geq \sum_{i=1}^p s^A_i||\mathbf{X}_i||^2 - \eta,$$ hence: $$\sum_{i=1}^p s^o_i||\mathbf{X}_i||^2 - \sum_{i=1}^p s^A_i||\mathbf{X}_i||^2 + \eta \geq \mathbf{V}^o - \mathbf{V}^A.$$ As $\eta \geq \eta(\mathbf{s}^o)$, $s^o$ is feasible in Problem $(7)$ so $$ \sum_{i=1}^p s^o_i||\mathbf{X}_i||^2 \leq \sum_{i=1}^p s^A_i||\mathbf{X}_i||^2,$$ then $\eta \geq \mathbf{V}^o - \mathbf{V}^A$.
\medbreak
We show now that the bound is tight. We build an example for $\mathbf{X}\in\mathbb{R}^{2 \times 4}$. Consider $0<\boldsymbol\epsilon < 0.5$ and $$\mathbf{X} = \big(\begin{smallmatrix}
  -\boldsymbol\epsilon & \boldsymbol\epsilon & 1 & 1\\
  1 & 1 & 0 & 0
\end{smallmatrix}\big)$$ and consider $k=2$ and $a=1$. The columns of $\mathbf{X}$ are illustrated in Figure \ref{fig1}. Let us first note that the value of the objective function of the optimal solution of problem $(2)$ is 2 and is given by taking $\mathbf{s}^* = (1\,1\,0\,0)$. If $\eta \geq 2 \boldsymbol\epsilon$, then $\mathbf{s}^*$ is feasible in Problem $(7)$ and the optimal solution is given by $\mathbf{s}^*$ and $\mathbf{U}[\mathbf{s}^*] = \mathbf{U}_1 = (0\,1)$. If $\eta < 2\boldsymbol\epsilon$, then it is impossible to have $s_1 = s_2 = 1$; the optimal solution is then $\mathbf{s}'=(0\,0\,1\,1)$ and the value of the objective function of problem $(2)$ when $\mathbf{s}=\mathbf{s}'$ is equal to $2-2\boldsymbol\epsilon$ as $\mathbf{U}[\mathbf{s}'] = \mathbf{U}_2 = (1\,0)$. Since $\eta$ can be taken as close as desired to $2\boldsymbol\epsilon$, then the bound is tight. 
\end{proof}
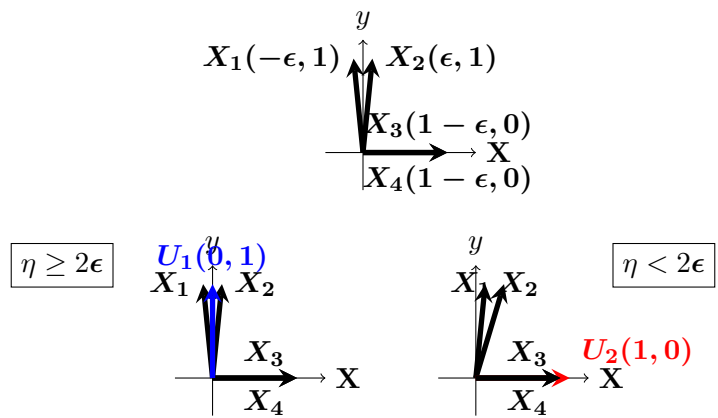
\begin{figure}
\centering
\begin{tikzpicture}
  \draw[->] (-0.5,3)--(1.5,3) node[right]{$\mathbf{X}$};
  \draw[->] (0,2.5)--(0,4.5) node[above]{$y$};
  \draw[line width=2pt,black,-stealth](0,3)--(-0.125,4.25) node[anchor= east]{$\boldsymbol{X_1(-\boldsymbol\epsilon,1)}$};
  \draw[line width=2pt,black,-stealth](0,3)--(0.125,4.25) node[anchor= west]{$\boldsymbol{X_2(\boldsymbol\epsilon,1)}$};
  \draw[line width=2pt,black,-stealth](0,3)--(1.125,3) node[anchor=south ]{$\boldsymbol{X_3(1-\boldsymbol\epsilon,0)}$};
  \draw[line width=2pt,black,-stealth](0,3)--(1.125,3) node[anchor=north ]{$\boldsymbol{X_4(1-\boldsymbol\epsilon,0)}$};
  
  \draw[->] (-2.5,0)--(-0.5,0) node[right]{$\mathbf{X}$};
  \draw[->] (-2,-0.5)--(-2,1.5) node[above]{$y$};
  \draw[line width=2pt,black,-stealth](-2,0)--(-2.125,1.25) node[anchor= east]{$\boldsymbol{X_1}$};
  \draw[line width=2pt,black,-stealth](-2,0)--(-1.875,1.25) node[anchor= west]{$\boldsymbol{X_2}$};
  \draw[line width=2pt,black,-stealth](-2,0)--(-0.875,0) node[anchor=south east]{$\boldsymbol{X_3}$};
  \draw[line width=2pt,black,-stealth](-2,0)--(-0.875,0) node[anchor=north east]{$\boldsymbol{X_4}$};
  \draw[line width=2pt,blue,-stealth](-2,0)--(-2,1.25) node[anchor=south]{$\boldsymbol{U_1(0,1)}$};
  \node[draw] at (-4,1.5) {$\eta \geq 2 \boldsymbol\epsilon$};
  
  \draw[->] (1,0)--(3,0) node[right]{$\mathbf{X}$};
  \draw[->] (1.5,-0.5)--(1.5,1.5) node[above]{$y$};
  \draw[line width=2pt,black,-stealth](1.5,0)--(1.875,1.25) node[anchor= east]{$\boldsymbol{X_1}$};
  \draw[line width=2pt,black,-stealth](1.5,0)--(1.625,1.25) node[anchor= west]{$\boldsymbol{X_2}$};
  \draw[line width=2pt,red,-stealth](1.5,0)--(2.75,0) node[anchor=south west]{$\boldsymbol{U_2(1,0)}$};
  \draw[line width=2pt,black,-stealth](1.5,0)--(2.625,0) node[anchor=south east]{$\boldsymbol{X_3}$};
  \draw[line width=2pt,black,-stealth](1.5,0)--(2.625,0) node[anchor=north east]{$\boldsymbol{X_4}$};
  \node[draw] at (4,1.5) {$\eta < 2 \boldsymbol\epsilon$};
\end{tikzpicture}
\caption{Example illustrating the tightness of Proposition 2's bound.} \label{fig1}
\end{figure}

This bound expresses the fact that the approximation is as good as the hypothesis that the matrix $\mathbf{X}(\mathbf{s}^o)$ composed of the columns of the optimal support can be approximated with a matrix of rank $a$. Even if $\eta^0$ is not known, we show in the implementation of \textbf{Algorithm 2} that this bound is practical. We prove first that a bound can be found \textit{a priori} by solving the classic PCA problem. We need first the following result:

\begin{proposition}
Consider that $\mathbf{X} = \mathbf{VV}^T\mathbf{X} + \boldsymbol\epsilon$ with $\mathbf{V}\in\mathbb{R}^{n\times a}$ and $\mathbf{V}^T\mathbf{V}=\mathbf{I}_a$. Let $(\mathbf{s}^o,\mathbf{U}[\mathbf{s}^o])$ be an optimal solution of problem $(2)$. We have:
$$ \eta(\mathbf{s}^o) \leq ||\mathbf{S}^o\boldsymbol\epsilon||^2.$$
Furthermore, $\mathbf{s}^o$ is feasible when $\eta = ||\mathbf{S}^o\boldsymbol\epsilon||^2$ in problem $(7)$.
\end{proposition}
\begin{proof}
We have: 
\begin{align*}
    ||\mathbf{S}^o\boldsymbol\epsilon||_F & = \,||\mathbf{S}^o (\mathbf{X} - \mathbf{VV}^T\mathbf{X})||_F\\
    & = \, ||\mathbf{S}^o \mathbf{X} - \mathbf{S}^o\mathbf{VV}^T\mathbf{X}||_F\\
    & = \, ||\mathbf{S}^o \mathbf{X} - \mathbf{VV}^T\mathbf{S}^o\mathbf{X}||_F \,\textrm{as}\, \mathbf{S}^o\mathbf{VV}^T = \mathbf{VV}^T\mathbf{S}^o\\
    & \geq \, ||\mathbf{X}(\mathbf{s}^o) - \mathbf{U}[\mathbf{s}^o]\mathbf{U}[\mathbf{s}^o]^T\mathbf{X}(\mathbf{s}^o)||_F \,\textrm{by definition}.\\
\end{align*}
When $\eta = ||\mathbf{S}^o\boldsymbol\epsilon||^2$, since $$||\mathbf{X}(\mathbf{s}^o) - \mathbf{U}[\mathbf{s}^o]\mathbf{U}[\mathbf{s}^o]^T\mathbf{X}(\mathbf{s}^o)||^2_F \leq ||\mathbf{S}^o\boldsymbol\epsilon||,$$ then $$||\mathbf{X}(\mathbf{s}^o) - \mathbf{U}[\mathbf{s}^o]\mathbf{U}[\mathbf{s}^o]^T\mathbf{X}(\mathbf{s}^o)||^2_F \leq \eta$$ and hence $s^o$ is feasible in Problem $(7)$.
\end{proof}
\begin{corollary}
Consider $\mathbf{V}^*\in argmin_{V\in\mathbb{R}^{n\times a}}||\mathbf{X}-\mathbf{VV}^T\mathbf{X}||\,st\,\mathbf{V}^T\mathbf{V}=\mathbf{I}_a$ (i.e., $\mathbf{V}$ is a solution to the classical PCA problem). We denote $\boldsymbol\epsilon = \mathbf{X} - \mathbf{V}^*\mathbf{V}^{*T}\mathbf{X}$ and $\sigma\subset [p]$ the set of the indices of $k$ columns of $\boldsymbol\epsilon$ with the highest norm. We have:
$$\eta(\mathbf{s}^o) \leq ||\mathbf{S}^o\boldsymbol\epsilon||^2_F \leq \sum_{i\in\sigma}||\boldsymbol\epsilon_i||^2,$$
\end{corollary}
 $\boldsymbol\epsilon$ can be easily computed using classic PCA. Although $\boldsymbol S^o$ is not known, the norm of the columns of $\boldsymbol\epsilon$ can be computed and so is $\sum_{i\in\sigma}||\boldsymbol\epsilon_i||$. Using the same notations as in Proposition 2 and Corollary 4, we have:
\begin{corollary}
The difference between the optimal value of problem $(2)$, $||\mathbf{U}[\mathbf{s}^o]\mathbf{U}[\mathbf{s}^o]^T\mathbf{X}(\mathbf{s}^o)||_F^2$, and the value of the objective function of problem $(2)$ evaluated at $s^A$, $||\mathbf{U}[\mathbf{s}^A]\mathbf{U}[\mathbf{s}^A]^T\mathbf{X}(\mathbf{s}^A)||_F^2$ is bounded by $\sum_{i\in\sigma}||\boldsymbol\epsilon_i||^2$.
\end{corollary}
\subsection{Algorithm}
We propose two algorithms. \textbf{Algorithm 1} solves Problem (7) for a given parameter $\eta$ while \textbf{Algorithm 2} starts without any knowledge on $\eta$ and then approaches the optimal solution of Problem (2) by exploring several values of $\eta$ while also updating the worst case upper bound.\medbreak

We propose an implementation based on cut generation (\textbf{Algorithm 1}). We initially solve the problem without any of the constraints $\sum_{i\in\sigma}s_i \leq |\sigma| - 1$, and then iteratively add these constraints. If at iteration $t$, the optimal solution found $\mathbf{s}^t$ violates $\eta(\mathbf{s}^t) \leq \eta$, then we add the constraint $\sum_{i\in\sigma^t}s_i \leq |\sigma^t| - 1$ such that $\sigma^t$ is the set of indices $i$ such that $s^t_i = 1$. Note that the computation of $\mathbf{U}[\mathbf{s}^t]$ is inexpensive as it involves only the generation of an SVD decomposition of the matrix $\mathbf{X}(\mathbf{s}^t)$ which is of size $n\times k$ which can in done in $\mathcal{O}(k^3 + nk^2)$. \medbreak

\begin{algorithm}[H]
\SetAlgoLined

\textbf{Input}: Data matrix $\mathbf{X}$, number of components $a$, number of variable sought $k$, parameter $\eta$\ and a set cuts (optional);
\textbf{Output}: Optimal support for the approximation formulation\;

 Initiate $(\phi)$, a BLO problem with the following formulation:
 \begin{align*}
    \max_{\boldsymbol s\in \{0,1\}^p} & \,\sum_{i=1}^ps_i||\mathbf{X}_i||^2\\
     s.t.\,&\mathbf{e}^T\mathbf{s} \leq k,\\
\end{align*}
Add the input set of cuts if there is any.
 Compute $\mathbf{s}^0$ an optimal solution of $(\phi)$ using a BLO solver\;
 Compute $\mathbf{U}[\mathbf{s}^0]$ by solving the PCA problem for matrix $\mathbf{X}(\mathbf{s}^0)$\;
 \While{$||\eta(\mathbf{s}^t) > \eta$ }{
  Update $(\phi)$ by adding the constraint $\sum_{i\in\sigma^t}s_i \leq |\sigma^t| - 1$\;
  Compute $\mathbf{s}^{t+1}$ an optimal solution of $(\phi)$ using a BLO solver\;
  Compute $\mathbf{U}[\mathbf{s}^{t+1}]$ by solving the PCA problem for matrix $\mathbf{X}(\mathbf{s}^{t+1})$\;
 }
 return the support found by solving $(\phi)$
 \caption{Constraints generation algorithm for PCs with a common support}
\end{algorithm}
While \textbf{Algorithm 1} terminates as $\{0,1\}^p$ is finite and the \textit{while} loop suppresses at least one element of $\{0,1\}^p$, the number of iterations is still potentially exponential. However, in practice the algorithm is relatively fast and solves large real world problems in minutes as illustrated in Section 4.\medbreak
We also use the proof of Theorem 4 to refine the search for $\eta(\mathbf{s}^o)$. We have:
\begin{lemma}
We assume that Problem $(7)$ has a unique optimal solution. The function $f:\eta\rightarrow f(\eta) $ returning the value of the objective function of problem $(2)$ evaluated at the optimal value of problem $(7)$ depending on $\eta$ is a piece wise constant function.
\end{lemma}
\begin{proof}
We note denote the function $F$ as $F(\eta) = \{\boldsymbol s\in \{0,1\}^p | \eta(\mathbf{s}) \leq \eta\}$. For a value $\eta \in \mathbb{R}^{+}$, $F$ returns the set of supports $\mathbf{s}$ such that $\eta(\mathbf{s}) \leq \eta.$ The cardinality of $F(\eta)$ is finite as $\{0,1\}^p$ is finite. The cardinality of $F(\eta)$ is increasing with respect to $\eta$ and $f$ is constant when $F$ is which proves the lemma.
\end{proof}
Using this lemma, we can derive a method to efficiently select $\eta$. Indeed, consider that we solve problem $(7)$ for a value $\Tilde{\eta}$ and note $\mathbf{\Tilde{s}}$, an optimal solution. It is easy to see that $f(\eta)$ is constant for $\eta \in [\,\eta(\mathbf{\Tilde{s}}),\Tilde{\eta}]$. We can then start with a large value of $\eta^0$ and iteratively update $\eta^t$ by computing $\eta(\mathbf{s}^t)$ where $s^t$ is a solution obtained and then input a new value $\eta^{t+1}=\eta(\mathbf{s}^t) - \delta$ for $\delta$ small enough as $f(\eta)$ is constant on $[\,\eta(\mathbf{s}^t),\eta^{t}]$.\medbreak
This also provides means to obtain tighter optimality gaps. Indeed, using Proposition 3 and starting with $\eta^0$, which is large enough, we can store and update $\eta^*$, the value of $\eta$ that achieves the largest value $f(\eta)$ and since we have started with values of $\eta$ larger than the $\eta(\mathbf{s}^o)$ corresponding to the optimal solution of problem $(2)$, then according to Proposition 3, the difference between the optimal solution of problem $(2)$ and $f(\eta^*)$ is capped by $\eta^*$.\medbreak
We report then the following algorithm to solve or approximate problem $(2)$ and obtain $\eta^*$.\medbreak
\begin{algorithm}[H]
\SetAlgoLined
\textbf{Input}: Data matrix $\mathbf{X}$, number of components $a$, number of variable sought $k$ and an integer $\lambda$ \;
\textbf{Output}: Support $s^*$ approximating the solution and $\eta^*$\;

 Initiate with $\eta^0$ and $\eta^*$ large enough\;
 c:=0\;
 \While{$c \leq \lambda$}{
  Run Algorithm 1 using $\eta^t$\ and cuts generated so far;
  \uIf{$f(\eta^*) < f(\eta^t)$}{
  $\eta^*$ := $\eta^t$\;
  }
  $\eta^{t+1}$ :=  $\eta(\mathbf{s}^t) - \delta$\;
  c:= c + 1\;
 }
 return $\mathbf{s}^*, \eta^*, f(\eta^*)$
 \caption{Approximate solution for problem $(2)$}
\end{algorithm}
\begin{proposition}
Algorithm 2 converges and for $\delta > 0$ small enough, Algorithm 2 finds an optimal solution to (2) in a finite number of steps for $\lambda$ large enough.
\end{proposition}
\begin{proof}
We first note that Algorithm 2 converges as $f(\eta^*)$ increases from an iteration to another by construction and is upper-bounded by $||X||_F^2$. \medbreak
We consider now $\mathbf{s}^o$ an optimal solution of (2), and let $\Delta = \{ |\eta(\mathbf{s}) - \eta(\mathbf{s'})| > 0,s.t. \mathbf{e}^T\mathbf{s} = \mathbf{e}^T\mathbf{s'} = k; \mathbf{s,s'}\in\{0,1\}^p\}$. Suppose that $\Delta$ is not empty. Consider any $\delta>0$ such that $\delta < min(\Delta)$. Considering Algorithm 2, we know that $\eta^0 > \eta(\mathbf{s}^o)$. If $f(\eta^0) = f(\eta(\mathbf{s}^o))$ then Algorithm 2 finds an optimal solution $\mathbf{s}^0$ at the first iteration as $f(\eta(\mathbf{s}^o))$ is an optimal solution to problem (2) according to Theorem 4. Consider now $t$ such that $\eta^t > \eta(\mathbf{s}^o)$ such that $f(\eta^t) < f(\eta(\mathbf{s}^o))$. We show now that $f(\eta^{t+1}) = f(\eta(\mathbf{s}^o))$ or $\eta^{t+1}> \eta^{o}$.  If $\eta^{t+1} < \eta(\mathbf{s}^o)$, then $\mathbf{s}^o$ is feasible in (7) and $\sum_{i=1}^ps^{t+1}_i||\mathbf{X}_i||^2\geq \sum_{i=1}^ps^{o}_i||\mathbf{X}_i||^2$. We also have by construction of $\delta$, $\eta(\mathbf{s}^{t+1}) \leq \eta(\mathbf{s}^o)$. Since $f(\eta(s)) = \sum_{i=1}^ps_i||\mathbf{X}_i|| + \eta(s)$, we have $f(\eta^{t+1}) \geq f(\eta(\mathbf{s}^o))$ and since $f(\eta(\mathbf{s}^o))$ is the optimal value of (2), then $f(\eta^{t+1}) = f(\eta(\mathbf{s}^o))$ meaning that Algorithm 2 finds an optimal solution $\mathbf{s}^{t+1}$. \medbreak
If $\eta^{t+1} = \eta(\mathbf{s}^o)$, then $f(\eta^{t+1}) = f(\eta(\mathbf{s}^o))$ according to Theorem 4. We conclude that if $f(\eta^{t+1})<f(\eta(\mathbf{s}^o))$ then $\eta^{t+1} < \eta(\mathbf{s}^o)$.\medbreak Hence, since $\eta^{t+1} - \eta^{t+2} > \delta > 0$ by construction, Algorithm 2 finds an optimal solution in less than $(\eta^0 - \eta(\mathbf{s}^o))/\delta$ steps. \medbreak
If $\Delta = \emptyset$, then an optimal solution of (7) is also an optimal solution to (2) as all $\eta(\mathbf{s})$ are equal and the Algorithm 2 finds the optimal solution at the first iteration.
\end{proof}

We note that the number of values $f(\eta)$ takes is lower than the number of feasible solution, which contributes to simplifying the problem. The condition $c \leq \lambda$ in the \textit{while} loop is a stopping criteria. It stops the search of a solution if no improvement is achieved after $\lambda$ attempts after $\eta^*$ is updated. Of course, other stopping criteria could be considered in this algorithm.\medbreak
We finally provide an additional upper-bound for the optimal solution of (2) that leverages on the history of search of Algorithm 2.
\begin{proposition}
Consider that Algorithm 2 generated $\theta$ cuts and note $s^\theta$ an optimal solution of (7) using the $\theta$ cuts generated. If $f(\eta^*) \geq \sum_{i=1}^ps^\theta_i||\mathbf{X}_i||^2$, then $s^*$ is an optimal solution for (2). Otherwise, the optimal solution of (2) is upper-bounded by $\sum_{i=1}^ps^\theta_i||\mathbf{X}_i||^2$.
\end{proposition}
\begin{proof}
We note $s^\theta$ a solution of (7) after generating $\theta$ cuts using Algorithm 2, $\phi(\theta)$ the value taken by the objective function of (2) at $s^\theta$ and $\psi(\theta) = \sum_{i=1}^ps^\theta_i||\mathbf{X}_i||^2$, the value objective function of (7) at the same point $s^\theta$. We first note that $\psi$ is a decreasing function. Indeed, the optimal value of (7) can only decrease when we add constraints. We also note that for any $\theta$, $\phi(\theta)\leq\psi(\theta)$. \medbreak 
If  $f(\eta^*)\geq \sum_{i=1}^ps^\theta_i||\mathbf{X}_i||^2 = \psi(\theta)$, then for any $\theta' \geq \theta$, $f(\eta^*) >= \psi(\theta')$ as $\psi$ is decreasing; and since $\phi(\theta')\leq\psi(\theta')$, then $f(\eta^*)\geq \psi(\theta')$ so $f(\eta^*)$ is optimal.\medbreak
Otherwise, by construction, $f(\eta^*) \geq \phi(\theta'')$ for any $\theta'' \leq \theta$. We also have for any $\theta' \geq \theta$, $\psi(\theta) \geq \psi(\theta')$ ans $\psi$ is decreasing and we also have  $\phi(\theta')\leq\psi(\theta')$ so $\psi(\theta) \geq \phi(\theta')$ so the optimal solution of (2) is indeed upper-bounded by $\sum_{i=1}^ps^\theta_i||\mathbf{X}_i||^2 = \psi(\theta)$.
\end{proof}
\subsection{Complexity assessment}
There are two aspects that need to be considered to assess the theoretical efficiency of the proposed algorithms; the number of cuts generated) and the computational cost of each iteration.\medbreak
Let us start by the number of cuts. As mentioned earlier, the number of cuts is potentially exponential. We provide here a theoretical example in which the number of cuts generated is exponential whatever the dimension $p$ chosen. We choose $k=2$ and $a=1$ and consider $\alpha = 5/24$ and $n = p$. Consider the matrix $\mathbf{X}\in\mathbb{R}^{n \times p}$ such that $x_{1,1} = 1$, $x_{1,2} = 1$, $x_{1,j} = 1 - \alpha$ for $j\geq 3$, $x_{i,i} = 2\alpha$ and $x_{i,i+1} = -2\alpha$ for $i\geq 3$, $x_{2,p} = -2\alpha$ and $x_{i,j} = 0$ otherwise. It is easy to verify that the norm of any column $\mathbf{X}_j$ is strictly greater than 1 for $j \geq 1$ and the norm of the two first columns is equal to 1. Is is also easy to verify that pair of columns $(i,j) \neq (1,2), (2,1)$, $\mathbf{U}^T\mathbf{X}(ij)^T\mathbf{X}(ij)\mathbf{U} < 2$ where $\mathbf{X}(ij)$ is a sub-matrix of $\mathbf{X}$ formed by its $i$ and $j$ columns and $\mathbf{U}\in\mathbb{R}^p$ with $||U|| = 1$. It is then easy to verify that the optimal solution $s^o$ of (2) is achieved with $s_1=s_2=1$ and $s_i = 0$ for $i\geq 3$. We set now $\eta = 0.00001$, then the only feasible solution for (7) is $s^o$. Since $||\mathbf{X}_i||^2 + ||\mathbf{X}_j||^2 > 2$ for any $(i,j) \neq (1,2), (2,1)$, and the cuts generated cut one binary point at a time, then Algorithm 1 will have to cut all the binary points except the optimal solution before it reaches the optimal solution.\medbreak
We examine now the cost of each iteration. Each iteration involves two steps (i) solving a BLO problem and (ii) a running a separation algorithm. As mentioned earlier, the separation algorithm consists of a classic PCA problem of dimension $n\times k$ which can be solved by using an SVD decomposition in $\mathcal{O}(k^3 + nk^2)$. Solving the BLO problem in Algorithm 1 is not trivial and the constraints matrix defines a polyhedron that has non-integer vertices in the general case. Yet, we can show that the BLO problem can be replaced by a much simpler algorithm. Indeed, at the first iteration of Algorithm 1 (no cut generated yet), the solution is obtained by choosing the $k$ columns of $\mathbf{X}$ that have the largest norms. Suppose now that we are at iteration $t$. If the separation algorithm finds a violated cut at iteration $t-1$, the new cut generated cuts exactly one binary point which is the optimal solution found of iteration $t-1$. This means that a solution of the BLO at iteration $t$ is defined by a set of vectors  with the largest sum of norms that are different from the solutions of iterations $1,2,3,..., t-1$. Since the BLO chooses the sets with largest sums of norms, a solution for the BLO of iteration $t$ has a sum of norms lower or equal to the sum of the binary solutions that were cut. Finding such solution could then be performed through tree search with a cost of $\mathcal{O}(k)$ as this process is then reduced finding the $k$ columns with the largest norms, then the set with the second largest sum of norms,..., then the set with the $t^{th}$ largest sum of norms. The actual computational cost of each iteration is then $\mathcal{O}(k^3 + nk^2)$.\medbreak
Since the cost of each iteration is modest, theoretically the potential issue would be with the number of iteration that could be theoretically exponential. Yet, in practice, high quality solutions are found in a tractable fashion for fairly large instances as we show in the following section using real-life data.\medbreak
In addition, it is worth noting that our method does not require computing the covariance matrix $\mathbf{X}^T\mathbf{X}$. Since we are interested in particular in instances in which $p$ is large (and in general high dimension setting we have $n\ll p$), this results in dramatic reduction in memory requirements in addition to reductions in computation time.\medbreak
Finally, although solving the BLO is computationally costly in theory, we have found that it is actually fast in practice when it comes to the BLO proposed in Algorithm 1. In Section 4, we illustrate how Algorithm 2 performs vs. other methods. We keep using the BLO because (i) solving the BLO using commercial solvers is so fast that Algorithm 2 already scales more than any other known method that tackles the problem considered in this section, (ii) the BLO could hold the opportunity to add tighter cuts (for potential future research) and (iii) too often, solving is deemed computationally costly while in practice it can be competitive, so we decided to illustrate the fact that considering BLO could be a valid approach to tackle high dimension problems.
\section{Extension to the Disjoint Support Block Sparse PCA and the Structured Sparse PCA problems}
We consider now the problem of maximizing the variance generated by $b$ sets of PCs such that the PCs of a same set share the same support while the supports of PCs of different sets are disjoint. We call this problem the Disjoint Support Block Sparse PCA (DSB SPCA) problem.

\subsection{DSB SPCA Formulation}
 We note $(k_i)_{i\in [b]}$ the desired cardinality of supports for each set and $(a_i)_{i\in [k]}$ the number of PCs in each set. Following the same method to formulate  problem $(3)$, the problem can be written as follows:
\begin{align*}
\tag{8}
    \max_{\mathbf{s}_i\in \{0,1\}^{p},(\mathbf{W}_i\in \mathbb{R}^{p\times a_i});i\in[b]} & \,\sum_{i=1}^b tr(\mathbf{W}_i^T\mathbf{S}_i\mathbf{X}^T\mathbf{XS}_i\mathbf{W}_i)\\
    s.t.&\,\sum_{j=1}^p s_{ij} \leq k_i,\forall i\in [b],\\
    & \mathbf{W}_i^T\mathbf{S}_i\mathbf{W}_i= I_{a_i},\forall i\in [b],\\
    & \mathbf{W}_i^T\mathbf{W}_i = I_{a_i},\forall i\in [b],\\
    & \sum_{i=1}^ps_{ij} \leq 1, \forall j \in [p].
\end{align*}
We first note that when $\forall i\in [b],a_i=1$, then formulation $(5)$ matches the sparse PCA problem with disjoint support as formulated in  \citep{ryan}. Following the same steps to construct problem $(7)$, we obtain the following approximation for $(5)$:
\begin{align*}
\tag{9}
    \max_{\mathbf{s}_i\in \{0,1\}^p;i\in[b]} & \,\sum_{i=1}^b\sum_{j=1}^ps_{ij}||\mathbf{X}_i||^2\\
     s.t.&\,\sum_{j=1}^p s_{ij} \leq k_i,\forall i\in[b],\\
    & \forall \sigma \subset [p], \forall i\in[b], ||\mathbf{X}(\mathbf{s}_i^{\sigma}) - \mathbf{U}[\mathbf{s}_i^{\sigma}]\mathbf{U}[\mathbf{s}_i^{\sigma}]^T\mathbf{X}(\mathbf{s}_i^{\sigma})||_F^2 > \eta\\
    &\Rightarrow \sum_{j\in\sigma}s_{ij} \leq |\sigma| - 1,  \\
    & \sum_{i=1}^ps_{ij} \leq 1, \forall j \in [p].
\end{align*}
\subsection{Worst Case Scenario Upper Bound}
We generalize Proposition 2 as follows:
\begin{proposition}
Let $(\mathbf{s}_l^o,\mathbf{U}[\mathbf{s}_l^o])_{l\in[b]}$ be an optimal solution of Problem (8). We note $$\eta(\mathbf{s}^o) =\sum_{l =1}^b||\mathbf{X}(\mathbf{s}_l^o) - \mathbf{U}[\mathbf{s}_l^o]\mathbf{U}[\mathbf{s}_l^o]^T\mathbf{X}(\mathbf{s}_l^o)||^2_F$$Consider $(\mathbf{s}_l^A)_{l\in[b]}$ an optimal solution of Problem (9) for $\eta = \sum_{l = 1}^b \eta_l \geq \eta(\mathbf{s}^o)$, then the difference between the optimal value of $(8)$,$$\sum_{l = 1}^b\sum_{i=1}^ps^o_{li}\sum_{j = 1}^{a_l}(\mathbf{X}_i^T\mathbf{U}[\mathbf{s}_l^o]_j)^2,$$ and the value of the objective function of $(8)$ evaluated at $(\mathbf{s}_l^A)_{l\in[b]}$, $$ \sum_{l = 1}^b \sum_{i=1}^ps^A_{li}\sum_{j = 1}^{a_l}(\mathbf{X}_i^T\mathbf{U}[\mathbf{s}_l^A]_j)^2$$ is bounded by $\eta$. Furthermore, this bound is tight.
\end{proposition}
\begin{proof}
Let $\mathbf{V}^o$ be the optimal value of Problem $(8)$, $(\mathbf{s}_l^A)_{l\in[b]}$ be a solution to  Problem $(9)$ and $$ \sum_{l = 1}^b \sum_{i=1}^ps^A_{li}\sum_{j = 1}^{a_l}(\mathbf{X}_i^T\mathbf{U}[\mathbf{s}_l^A]_j)^2,$$ be the value of the objective function of Problem (8) evaluated at $(\mathbf{s}_l^A)_{l\in[b]}$. We have $$\sum_{l = 1}^b\sum_{i=1}^p s^o_{li}||\mathbf{X}_i||^2 \geq \mathbf{V}^o$$ and $$ \mathbf{V}^A \geq \sum_{l = 1}^b\sum_{i=1}^p s^A_{li}||\mathbf{X}_i||^2 - \eta.$$ Hence, $$\sum_{l = 1}^b\sum_{i=1}^p s^o_{li}||\mathbf{X}_i||^2 -\sum_{l = 1}^b \sum_{i=1}^p s^A_{li}||\mathbf{X}_i||^2 + \eta \geq \mathbf{V}^o - \mathbf{V}^A.$$ $(\mathbf{s}_l^o)_{l\in[b]}$ is feasible in Problem $(9)$ so $\sum_{l = 1}^b \sum_{i=1}^p s^o_{li}||\mathbf{X}_i||^2 \leq \sum_{l = 1}^b\sum_{i=1}^p s^A_{li}||\mathbf{X}_i||^2$, then $\eta \geq \mathbf{V}^o - \mathbf{V}^A$.
\end{proof}
Proposition 3 and corollaries 4 and 5 are still applicable in the case of disjoint block supports and the proofs are almost identical.
\subsection{Algorithm}
\textbf{Algorithm 1} can be easily adapted for the case of groups of PCs having disjoint supports
\medbreak
\begin{algorithm}[H]
\SetAlgoLined
\textbf{Input}: Data matrix $\mathbf{X}$, number of components $(a_l)_{l\in [b]}$, number of variables sought $(k_l)$ and parameter $(\eta_l)_{l\in [b]}$ per group of PCs.\;
\KwResult{Optimal supports for the disjoint supports approximation formulation}

 Initiate the following BLO problem $(\phi)$:
 \begin{align*}
    \max_{\mathbf{s}_i\in \{0,1\}^p;i\in[b]} & \,\sum_{i=1}^b\sum_{j=1}^ps_{ij}||\mathbf{X}_i||^2\\
     s.t.&\,\sum_{j=1}^p s_{ij} \leq k_i,\forall i\in[b],\\
    & \sum_{i=1}^ps_{ij} \leq 1, \forall j \in [p].
\end{align*}
 Compute $(\mathbf{s}^o)$ an optimal solution of $(\phi)$ using a BLO solver\;
 Compute $\mathbf{U}[\mathbf{s}^o]$ by solving the PCA problem for matrix $\mathbf{X}(\mathbf{s}^o)$\;
 \While{$\exists l \in [b]\,||\mathbf{X}(\mathbf{s}_l^t) - \mathbf{U}[\mathbf{s}_l^t]\mathbf{U}[\mathbf{s}_l^t]^T\mathbf{X}(\mathbf{s}_l^t)||^2_F > \eta_l$ }{
  Update $(\phi)$ by adding the constraint $\sum_{i\in\sigma^t}s_{li} \leq |\sigma^t| - 1$\ for all $l\in [b]$ such that $||\mathbf{X}(\mathbf{s}_l^t) - \mathbf{U}[\mathbf{s}_l^t]\mathbf{U}[\mathbf{s}_l^t]^T\mathbf{X}(\mathbf{s}_l^t)|| > \eta_l$;
  Compute $(\mathbf{s}_l^{t+1})_{l\in[b]}$ an optimal solution of $(\phi)$ using a BLO solver\;
  Compute $(\mathbf{U}[\mathbf{s}_l^{t+1}])_{l\in[b]}$ by solving the PCA problem for matrix $(\mathbf{X}(\mathbf{s}_l^{t+1}))_{l\in[b]}$\;
 }
 return the support found by solving $(\phi)$
\caption{Constraints generation algorithm for group of PCs with disjoint supports}
\end{algorithm}
\subsection{Structured Sparse PCA}
In some applications of sparse PCA, additional properties to sparsity are needed either to further improve interpretability or to enhance performances in subsequent classification or regression. One of the notable applications of Structured Sparse PCA is image recognition in which the fact that the features selected need to be adjacent and/or form a particular 2D pattern \citep{bach}. Another notable similar application is protein complex dynamics in which practitioners require that the 3D distance between the features to be limited and more recently, more abstract structures have been considered in genomics based on the interaction among different genes \citep{sspcabio}.\medbreak
The approximation $(7)$ can be adapted to virtually any pattern that can be defined by linear constraints. We propose a general formulation for the approximation Structured Sparse PCA problem following the same approach that lead us to propose problem $(7)$ and problem $(9)$. We then illustrate this formulation for 2D data in Section 3.\medbreak
Consider $\Pi=\{\pi_1,...,\pi_{|\Pi|}\}$ a set of subsets of $[p]$ that represent the patterns that are desired and $b>0$ a number of patterns that would be used to construct the PCs. In practice, patterns could represent a structure that is sought in the data; for instance, patterns could be related genes in genomics or 2D shapes in image recognition. We propose the following exact formulation:
\begin{align*}
\tag{10}
    \max_{\boldsymbol s\in \{0,1\}^p, \mathbf{z}\in \{0,1\}^{|\Pi|},\mathbf{W}\in \mathbb{R}^{p\times a}} & \,tr(\mathbf{W}^T\mathbf{S}\mathbf{X}^T\mathbf{X}\mathbf{S}\mathbf{W})\\
    s.t.&\,s_i \leq \sum_{j\in[\Pi]:i\in\pi_j}z_j,\forall i\in [p],\\
    &\, \sum_{j=1}^{|\Pi|} z_j \leq b,\\
    & \mathbf{W}^T\mathbf{W} = \mathbf{I}_a.
\end{align*}
Following the same steps as the ones we used to build the approximation $(7)$, we propose the following approximation for the structured sparse PCA problem for a given $\eta_{\tau}>0$:
\begin{align*}
\tag{11}
    \max_{\boldsymbol s\in \{0,1\}^p} & \,\sum_{i=1}^ps_i||\mathbf{X}_i||^2\\
    s.t.&\,s_i \leq \sum_{j\in[\Pi]:i\in\pi_j}z_j,\forall i\in [p],\\
    &\, \sum_{j=1}^{|\Pi|} z_j \leq b,\\
    & \forall \pi_j \in \Pi, ||\mathbf{X}(\mathbf{s}^{\pi_j}) - \mathbf{U}[\mathbf{s}^{\pi_j}]\mathbf{U}[\mathbf{s}^{\pi_j}]^T\mathbf{X}(\mathbf{s}^{\pi_j})||_F^2 > \eta_{\tau},\\
    &\Rightarrow z_j = 0.  \\
\end{align*}
This formulation can also be extended to the case in which the patterns are disjoint:
\begin{align*}
\tag{12}
    \max_{\boldsymbol s\in \{0,1\}^p} & \,\sum_{i=1}^ps_i||\mathbf{X}_i||^2\\
    s.t.&\,s_i \leq \sum_{j\in[\Pi]:i\in\pi_j}z_j,\forall i\in [p],\\
    &\, \sum_{j=1}^{|\Pi|} z_j \leq b,\\
    & \forall \pi_j \in \Pi, ||\mathbf{X}(\mathbf{s}^{\pi_j}) - \mathbf{U}[\mathbf{s}^{\pi_j}]\mathbf{U}[\mathbf{s}^{\pi_j}]^T\mathbf{X}(\mathbf{s}^{\pi_j})|| > \eta_{\tau}\\
    &\Rightarrow z_j = 0,  \\
    &\forall i \in [p]\, \sum_{j:i\in \pi^j} z_j \leq 1.\\
\end{align*}
(12) being a special case of  problem $(9)$ (each pattern representing one group of PCs sharing the same support), optimality bounds and algorithm apply in this case for $\eta_l = \eta_{\tau}$. Furthermore, if $|\Pi|$ is not too large, all patterns $\pi^j$ that violate $||\mathbf{X}(\mathbf{s}^{\pi_j}) - \mathbf{U}[\mathbf{s}^{\pi_j}]\mathbf{U}[\mathbf{s}^{\pi_j}]^T\mathbf{X}(\mathbf{s}^{\pi_j})|| \leq \eta_{\tau}$ can be enumerated and eliminated before solving the problem. In this case, all the constraints of Problem (12) can be enumerated and and there is no need to use a cut generation algorithm.
\section{Results}
We aim in this section to illustrate the benefits of Geometric Sparse PCA (GeoSPCA) method in terms of variance explained, sparsity, predictive power and tractability. We also aim at comparing the performances obtained when building all the PCs at once vs. building them iteratively by deflating the data matrix.\medbreak

We first explicit principles for the choice of $a$ and tuning of $\eta$. We then test and compare GeoSPCA in the case of all PCs sharing the same support on real world data sets. We finally test GeoSPCA on a image recognition data set using its structured sparsity version including disjointedness constraints of  problem $(9)$.\medbreak

All tests are conducted computations on an Intel Core i7-8750H CPU at 2.20GHz with 16Gb of RAM on Windows 10 Pro. The solver we used is Gurobi Optimizer 9.1 running with Python 3.6.5.

\subsection{Choosing $\eta$ and $a$}
We base the approach we use to choose $a$ and $\eta$ on the the similarities GeoSPCA has with the classic PCA. \medbreak We first tune $a$ by finding a suitable number of PCs for the classic PCA following a standard procedure. Namely, $a$ could be chosen as the number of PCs beyond which the marginal gain in explained variance is limited (Figure \ref{fig2}a). If $\mathbf{X}$ can be approximated by a matrix $\mathbf{X}'$ of rank $a$, then a sub-matrix $\mathbf{X}(\mathbf{s})$ of $\mathbf{X}$ composed of selected columns of $\mathbf{X}$ can be approximated by a matrix of rank $a$. Although, theoretically, $\mathbf{X}(\mathbf{s})$ could be approximated by a matrix of a rank lower than $a$, this is could be the case in a real data setting but this would involve a special structure in the data; in particular, a subset of the columns of $\mathbf{X}'$ must be orthogonal to the remaining columns, or in other words, independence among variables would be required. \medbreak The parameter $\eta$ is found using Algorithm 2. We start by choosing $\eta^0$ that is large (for example $||\mathbf{X}||_F^2$) and then Algotrithm 2 tightens the values of $\eta$ by updating $\eta^t$. When $\eta^t>\eta(\mathbf{s}^o) +\delta$ for a $\delta > 0$ (see Theorem 4), then $f(\eta^*)$ is lower than the optimal value of (2). When $\eta^t > \eta(\mathbf{s}^o)$ then the optimal solutions of problem $(2)$ are cut from the feasible set of problem $(7)$ (Figure \ref{fig2}.b).

\begin{figure}
\centering
    \includegraphics[ scale=0.18]{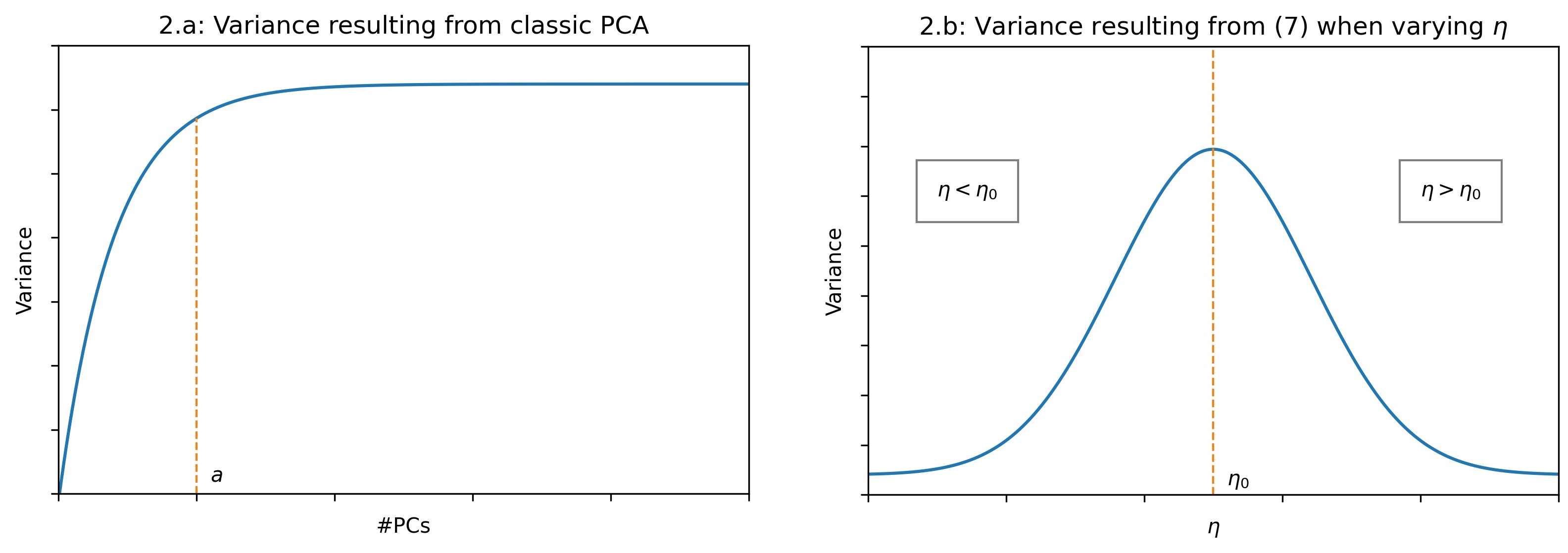}
    \caption{Illustrative figures for the choice of $a$ and $\eta$.}
    \label{fig2}
\end{figure}

\subsection{GeoSPCA with common support}
We consider the problem of maximizing the variance explained from $\mathbf{X}$ using a number $a$ of orthogonal PCs that have a common support of cardinal $k$. We use the formulation $(7)$ to approximate problem $(2)$. We use real world data sets of various natures and sizes to illustrate the benefits of GeoSPCA. In this subsection, we focus on the amount of explained variance, the benefit of building all PCs at once and the tractability depending on the size of the data set, the number of PCs and the cardinality $k$ chosen. We selected publicly available data sets that are widely used in the literature including \textit{Mturk} $(n;p) = (180;500)$ \citep{Mturk} which consists of descriptions of randomly chosen pictures using a bag of words, \textit{Colon} $(n;p) = (62;2,000)$ \citep{Alon} a colon cancer gene expression data set, \textit{Arcene} $(n,p) = (700;10,000)$ a mass-spectrometric data aiming at detecting cancer patterns proposed at the NIPS 2003 Feature Selection Challenge \citep{nips2003} and \textit{CGD} $(n;p) = (286;22,283)$ \citep{wang}, a gene expression data set used to classify breast cancers. 
\medbreak
We compare GeoSPCA to other techniques that control the sparsity by imposing $k$, the number of variables that are used in the sparse model. Comparison with other techniques in which sparsity is induced by regularization is not significant  because achieving the level of sparsity that is achieved while controlling $k$ requires to choose a very large weight for the regularization which skews the objective function and produces ultimately low quality solutions. Since we are also aiming at choosing $k$ higher than 100 and $p$ in 10,000s; we choose \textit{PathSPCA} \citep{aspremont2} to illustrate the benefits of building all PCs at once as its performance is comparable to other techniques extracting iteratively \citep{papailiopoulos13,gpowert}. Since this technique extracts one PC at each iteration, a deflation technique is needed. We choose Schur complement deflation for its empirical performance and ease of use (see \citep{deflation} for a description and a full discussion on deflation techniques). Sparse PCs found with this process are then orthogonalized. We also use a greedy algorithm directly inspired from \citep{moghaddam} and \citep{aspremont2}. We build the support of the solution $\sigma$ by iteratively including the indices that maximize the increase in variace pactured from an iteration to another. We start with an empty set and we iteratively add indices $i$ to $\sigma$ such that:

\begin{align*}
    i\in argmax_{i\in[p]\backslash \sigma}\displaystyle\max_{\mathbf{W}\in \mathbb{R}^{p\times a}} \,& tr(\mathbf{W}^T\mathbf{S}^{\sigma \cup \{i\} }\mathbf{X}^T\mathbf{X}\mathbf{S}^{\sigma \cup \{i\} }\mathbf{W})\\
    s.t.\,& \mathbf{W}^T\mathbf{W} = \mathbf{I}_a.
\end{align*}

\medbreak Greedy approaches have proven to be particularly effective and are often at par with state-of-the-art techniques \citep{aspremont2, ryan,journee,papailiopoulos13}.\medbreak
We report the sorted norms of the columns on $\mathbf{X}$ for the different data sets considered in Figure \ref{fig3} \textit{(Top)}. Formulation $(2)$ sheds a light on one reason the greedy approach performs well. Indeed, as explained in problem $(2)$, the objective function is $\sum_{i=1}^ps_i\sum_{j = 1}^a(\mathbf{X}_i^T\mathbf{U}_j)^2$, so maximizing the variance captured from $k$ features is closely related to the norm of the $k$ columns related to these features. If some columns of $\mathbf{X}$ have a much higher norm than the rest of the columns, they are likely to be selected by the greedy algorithm and their related variables are also likely to be in the support of the optimal solution. This is true for $a=1$ which is the case that is the most studied in the literature and many of the data sets considered in the literature have columns with a norm far exceeding the norms of the remaining columns.\medbreak

\begin{figure}
\centering
    \includegraphics[ scale=0.5]{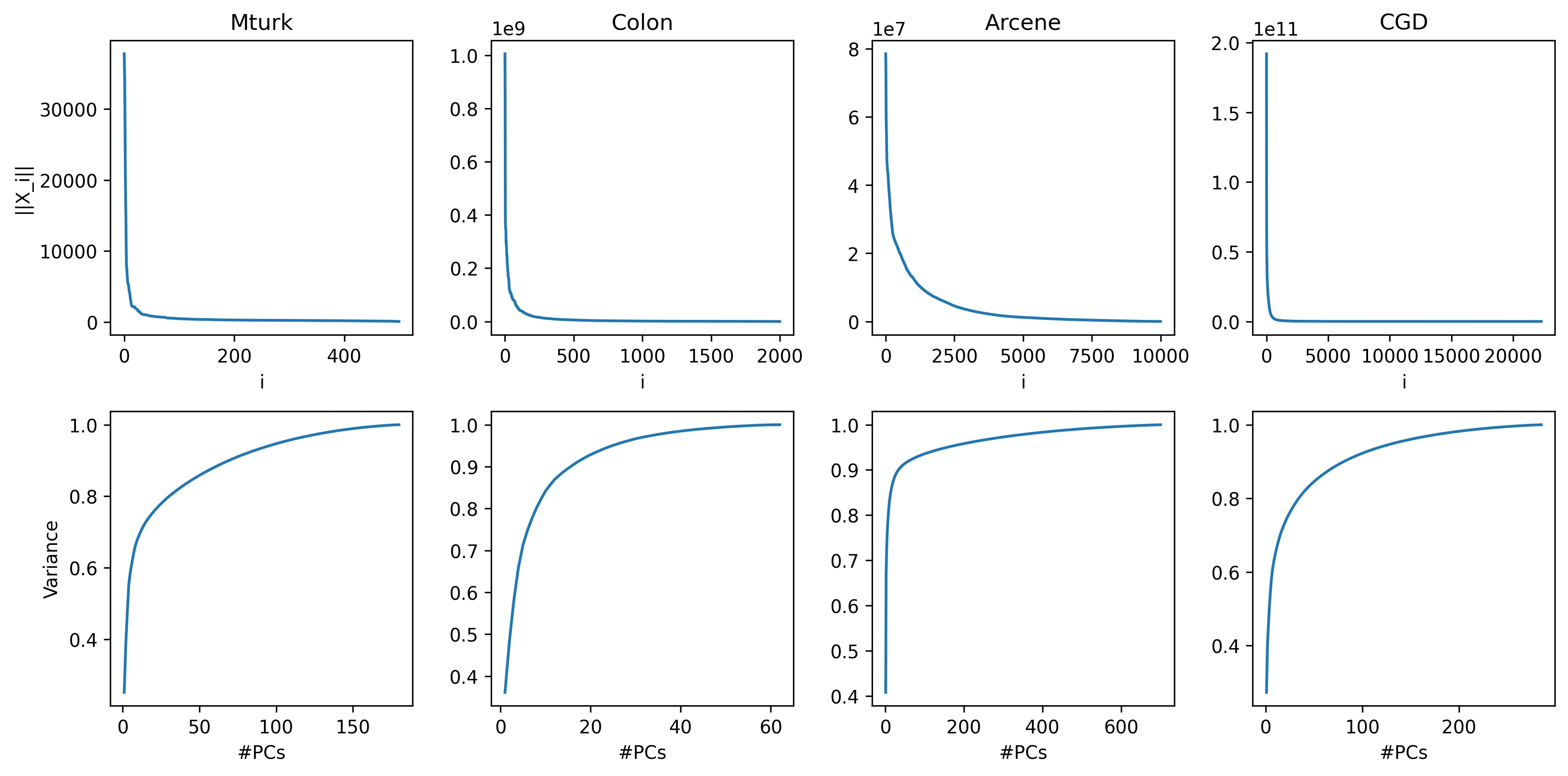}
    \caption{\textit{(Top)} Norms of the columns of $\mathbf{X}$ in decreasing order \textit{(Bottom)} Variance explained by standard PCA depending on the number of PCs }
    \label{fig3}
\end{figure}

We first conduct tests with \textit{PathSPCA} using a constant number of features for the construction of each PC. After generating $a$ PCs, we count the total number of features used $k$ (we note that different PCs may use common features). We then conduct tests using \textbf{Algorithm 1} using $k$ variables. Values of $k$ are chosen for illustration purposes and are also the result of the number of variables that the deflation method produces. Indeed, in all cases studied in this paper, \textit{PathSPCA} often chooses to reuse variables that were used in previously constructed PCs. The cuts are generated as lazy constraints.

\begin{table}

\centering
\caption{Explained Variance captured by each approach, GeoSPCA algorithm computation time, the optimality gap (GAP) and the number of cuts generated. $a$ is the number of PCs used. $k$ is the number of variables in the support. The columns \textit{Deflation}, \textit{Greedy} and \textit{GeoSCPA} indicate the variance obtained using \textit{PathPCA} combined with Schur complement deflation technique, the Greedy algorithm and \textit{GeoSPCA} respectively.  The \textit{GeoSPCA vs. Greedy} column shows the relative improvement in captured variance of \textit{GeoSPCA} vs. the Greedy algorithm. The \textit{Time} column indicates the computation time of \textit{GeoSPCA} Algorithm. GAP indicates the theoretical upper bound of the relative gap between the variance captured by \textit{GeoSPCA} vs. the variance captured by the optimal solution of the exact formulation as per propositions 5 and 11 ($0\%$ gap means that the solution is optimal as per Proposition 11). }

\begin{tabular}{ |c|c|c|c|c|c|c|c|c|c|  }
 \hline
Data set&$a$&k&Deflation&Greedy&GeoSPCA& \shortstack{GeoSPCA \\ vs. \\Greedy} & \shortstack{Time} & GAP & \#cuts \\
 \hline
  \hline

Mturk&8&16&5.62E+3&1.57E+5&1.59E+5&+1.1\%&$<$1s&0\%&56\\
$(n,p) =$&&21&5.6E+3&1.66E+5&1.67E+5&+1\%&$<$1s&1.7\%&44\\
(180;500)&&26&8.73E+3&1.71E+5&1.74E+5&+1.7\%&$<$1s&2.8\%&30\\
&&29&8.7E+3&1.76E+5&1.78E+5&+0.8\%&$<$1s&3.4\%&154\\
&&32&9.06E+3&1.8E+5&1.81E+5&+0.4\%&$<$1s&4\%&1208\\
\hline
Colon&5&11&1.02E+9&4.57E+9&4.79E+9&+4.9\%&$<$1s&1.7\%&120\\

$(n,p) =$ &&12&1.97E+9&4.74E+9&4.92E+9&+3.8\%&$<$1s&3.8\%&85\\
(62;2,000)&&15&2.27E+9&5.41E+9&5.49E+9&+1.5\%&$<$1s&8.4\%&562\\
&&18&2.8E+9&5.9E+9&5.94E+9&+0.6\%&$<$1s&12\%&384\\
&&33&3.81E+9&7.62E+9&7.6E+9&-0.3\%&$<$1s&21.2\%&1214\\
\hline
Arcene&3&14&7.43E+7&8.2E+8&1.02E+9&+24.2\%&$<$1s&0\%&6\\

$(n,p) =$&&22&8.7E+7&1.27E+9&1.5E+9&+18.1\%&$<$1s&0\%&9\\
(700;&&53&1.76E+8&2.74E+9&3.01E+9&+9.7\%&$<$1s&1.6\%&7\\
10,000)&&132&5.9E+8&5.83E+9&6.08E+9&+4.4\%&$<$1s&4.1\%&35\\
&&270&7.95E+8&9.78E+9&9.82E+9&+0.4\%&45s&7.1\%&590\\
\hline
CGD&11&23&5.11E+10&1.84E+12&1.89E+12&+2.7\%&$<$1s&1\%&71\\

(n,p) =&&43&1.09E+11&2.5E+12&2.63E+12&+5.1\%&$<$1s&6.3\%&50\\
(286;&&58&1.91E+11&2.89E+12&3E+12&+3.9\%&6s&10.2\%&166\\
22,283)&&85&2.82E+11&3.45E+12&3.54E+12&+2.9\%&20s&14.4\%&359\\
&&174&3.25E+11&4.64E+12&4.67E+12&+0.6\%&118s&23\%&1308\\

 \hline
\end{tabular}
\label{table1}
\end{table}

We report results in Table \ref{table1}. We first notice that deflation produces significantly lower explained variance vs. GeoSPCA and the greedy in all cases; sometimes by more than an order of magnitude. The illustrates the benefit of constructing all the PCs at once instead of doing so iteratively through deflation. Considering the Greedy approach, GeoSPCA outperforms this method in almost all cases by up to $24.2\%$. When it comes to the variance explained by sparse PCA, even $1\%$ is significant as this implies, depending on the application considered, more lives saved or increased profits. Even theoretically, considering problem $(2)$, when the vectors with the largest norms have a norm far exceeding the remaining columns (Figure 3), improving the explained variance by 1\% is remarkable. We notice also that for Arcene and CDG data sets, the error $\boldsymbol\epsilon$ has a greater norm compared to other data sets with respect to $a$. This signals that the sparse PCA problem is harder to solve which explains the edge GeoSPCA has over a simple greedy approach.\medbreak

Computation time for GeoSPCA did not exceed 2 mins while PathSPCA and the greedy approach needed several hours (and often tens of hours on CGD) to provide a solution (see Table \ref{tabletime}. We note also that GeoSPCA does not need to compute or handle the covariance matrix $\mathbf{X}^T\mathbf{X}$ which also contribute to drastically reduced computation time, especially for the largest instances. $\mathbf{X}^T\mathbf{X}$ has 4.84 10E9 entries that, besides the inherent PathSPCA computation time, needs to be deflated to compute each PC. This already creates complications in the management of the memory capacity of most personal computers.\medbreak
\begin{table}

\centering
\caption{Computation time range for $PathSPCA$ and Greedy Algorithms} 

\begin{tabular}{ |c|c|c|c|c|  }
 \hline
Method &Mturk&Colon&Arcene&CGD \\ 
 \hline

PathSPCA & 1 to 3s & 9 to 11s & 4 to 8 mins & 2 mins to 11 hours \\
Greedy & 2 to 10s & 4 to 20s & 2min to 8 hours & 4 mins to 51 hours \\

 \hline
\end{tabular}
\label{tabletime}
\end{table}
Still considering the same data sets and $a$ and $k$ values, we report in Figure \ref{Values} the evolution of the objective function of (2) in blue and (7) in orange as cuts are iteratively added in Algorithm 2 with respect to the number of cuts for the 3000 first cuts generated. For this experiment, we drop the lazy constraints feature in Gurobi. An optimal solution is found when the lowest value of (7) (in orange) is lower than a value of (2) (in blue). The green dotted line represents the lowest value of (7) achieved.\medbreak

\begin{figure}
\centering
    \includegraphics[ scale=0.13]{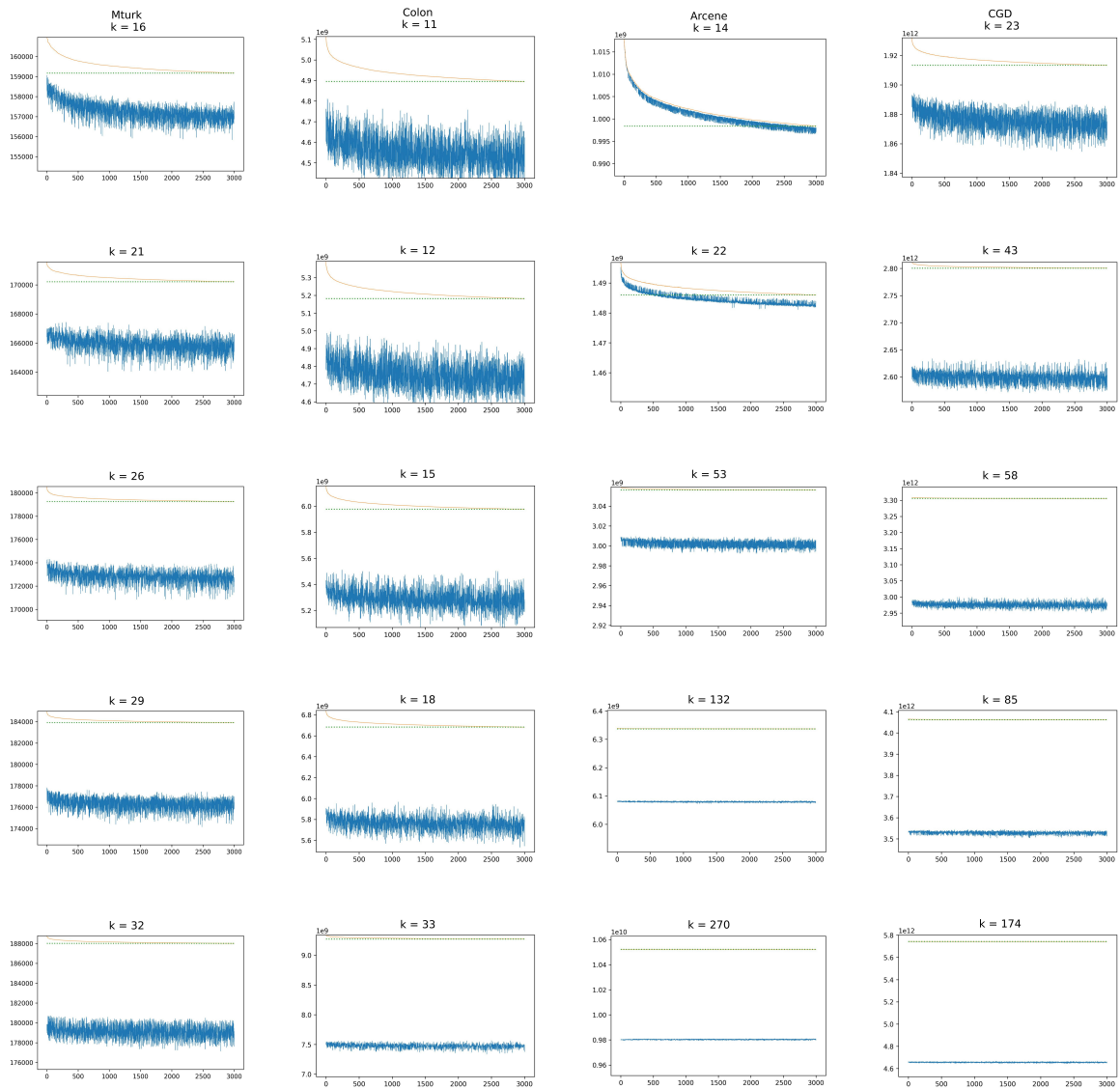}
    \caption{Values of the objective functions of (2) (in blue) and (7) (in orange) in function of the number of cuts added using Algorithm 2.}
    \label{Values}
\end{figure} 
We also report for the same experiment $f(\eta^t)$ with respect to $t$ in Figure \ref{feta}. We chose instances in which the optimal value is found by Algorithm 2 (using Proposition 11 to prove it). For Mturk (left), we notice that Algorithm 2 reduces $\eta^t$ until the optimal solution is found at $t = 2$, then optimal solutions are cut and $f(\eta^t)$ then decreases below the optimal solution value. For Arcene (right), the optimal solution is found at the first iteration, then the optimal solution is cut and $f(\eta^t)$ decreases also to values below the optimal value.

\begin{figure}
\centering
    \includegraphics[ scale=0.18]{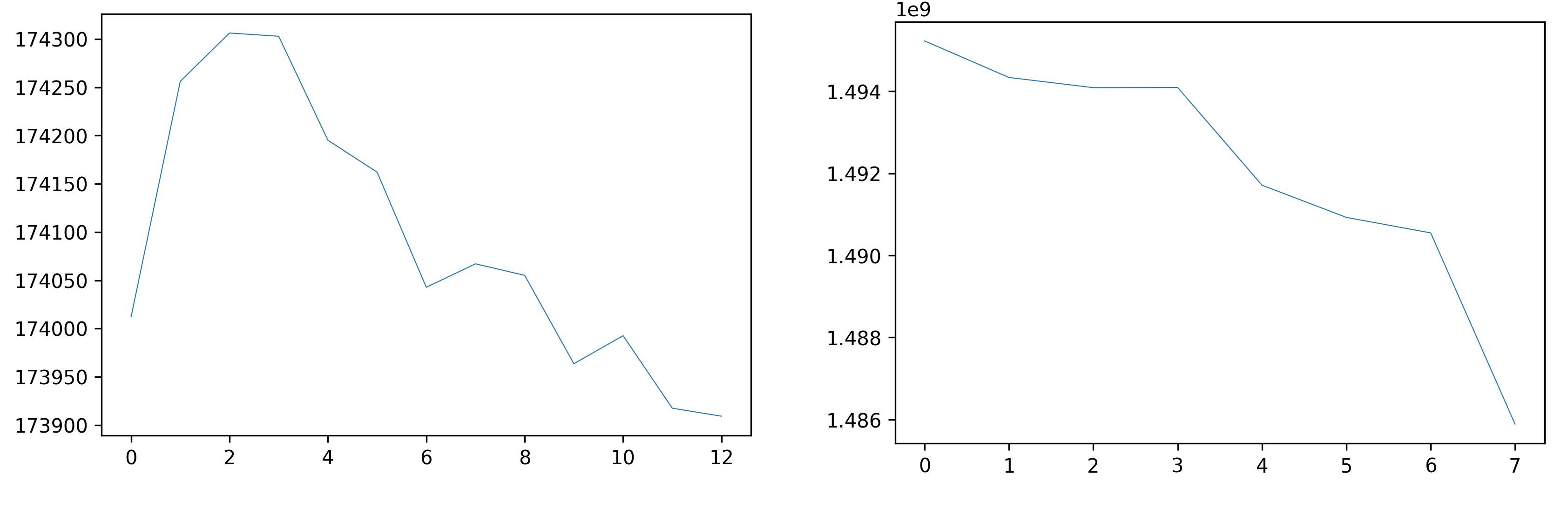}
    \caption{$f(\eta^t)$ with respect to $t$ on the left for Mturk data set ($(a,k) = (8,16)$) and Arcene ($(a,k) = 3,22)$)}
    \label{feta}
\end{figure}

\subsection{Experiments for Structured Sparse PCA}
Imposing a structure in the variables to build the PCs can yield substantial benefits vs. PCA or sparse PCA in a number of applications including genomics and face recognition\citep{sspcabio,bach}. We test GeoSPCA in its structured version (12) (we will call it GeoSSPCA). We chose to use face recognition  for its ease of interpretation to test the method and use the data set \citep{faces}. The data set consists of 2600 cropped pictures of the faces of 50 men and 50 women. For each person, 26 pictures are provided with the different face expressions and lightning configurations. In 12 of the 26 pictures, parts of the face is hidden either by black glasses or scarves (a sample of pictures is provided in Figure\ref{fig5}).
\begin{figure}
\centering
    \includegraphics[ scale=0.4]{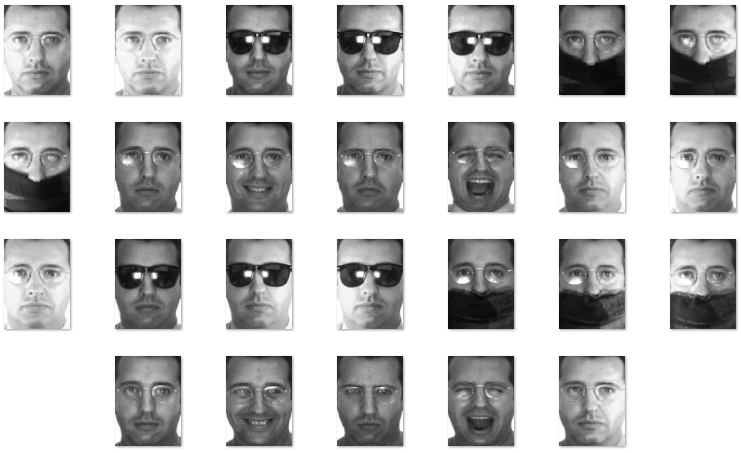}
    \caption{Sample of pictures of the data set chosen}
    \label{fig5}
\end{figure} 
We use as patterns triangles, rectangles and octagons with dimensions varying from 3 to 8 pixels as elements of the patterns set. We filter all patterns $\pi$ that violate the constraint $||\mathbf{X}(\mathbf{s}^{\pi_j}) - \mathbf{U}[\mathbf{s}^{\pi_j}]\mathbf{U}[\mathbf{s}^{\pi_j}]^T\mathbf{X}(\mathbf{s}^{\pi_j})||_F^2 \leq \eta_{\tau}$ (using the same notations as in (12)) so we solve only a BLO problem once as all remaining patterns verify this constraint. We choose $a=3$ and $k=7$ and report an example of PCs constructed while solving (12) in Figure\ref{fig6}. We notice that the shapes and the location of the shape capture intuitive components of the face most of the time (eyes, mouth, shape of the jaw and forehead,...).
\begin{figure}
\centering
    \includegraphics[ scale=0.5]{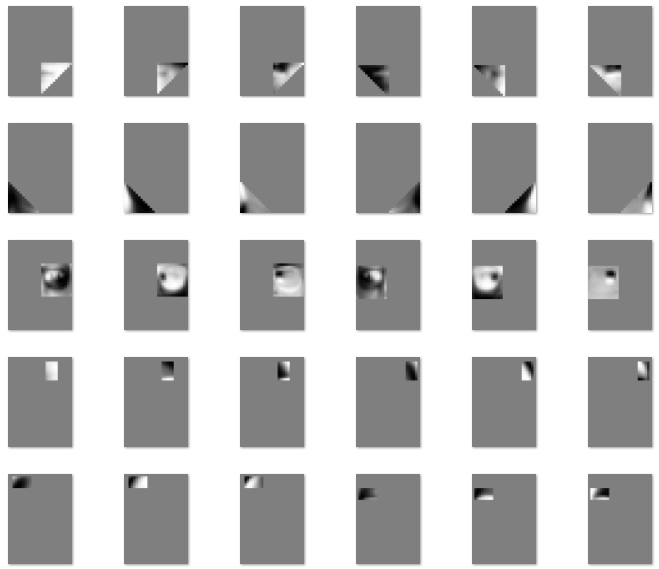}
    \caption{Example of PCs constructed for $k = 10$ and $a = 3$. First row shows 2 halves of the mouth; second row represents the shape of the jaw; third row the eye and the glasses (many participants to the data set were wearing glasses); fourth and fifth rows focus on the shape of the forehead and some on the eyebrows.}
    \label{fig6}
\end{figure}

\medbreak
We follow the same procedure used in \citep{bach}; the resolution of the images is reduced from $165\times 120$ to $38\times 27$; for each person represented in the data set, we use the 14 pictures in which the whole face is visible for training and test on the 12 pictures in which part of the face is hidden. We also use k-NN algorithm to classify the pictures after reducing the dimension using GeoSSPCA and compare to the precision achieved by Structured Sparse PCA\citep{bach} (using the modeling scheme proposed by the authors), and PCA depending on the number of PCs. For GeoSSPCA, we choose $k=5$ and increase $a$ to have a number of PCs varying from 10 to 70. A comparison of the precision is provided in Figure \ref{fig7}.\medbreak 
\begin{figure}
    \centering
    \begin{tikzpicture}
\begin{axis}[
    xlabel={\#PCs},
    ylabel={Precision},
    xmin=9, xmax=71,
    ymin=0, ymax=0.7,
    xtick={10,20,30,40,50,60,70},
    ytick={0,0.1,0.2,0.3,0.4,0.5,0.6,0.7},
    legend pos=north west,
    ymajorgrids=true,
    grid style=dashed,
]

\addplot[
    color=blue,
    mark=square,
    ]
    coordinates {
    (10,0.23416)(20,0.4625)(30,0.4958)(40,0.5116)(50,0.5516)(60,0.5875)(70,0.6166)
    };
    \legend{GeoSSPCA}

\addplot[
    color=orange,
    mark=*,
    ]
    coordinates {
    (10,0.095)(20,0.176)(30,0.233)(40,0.39)(50,0.358)(60,0.43)(70,0.49)
    };
    \legend{GeoSSPCA,SSPCA}

\addplot[
    color=black!30!green,
    mark=triangle,
    ]
    coordinates {
    (10,0.072)(20,0.17)(30,0.241)(40,0.39)(50,0.31)(60,0.351)(70,0.37)
    };
    \legend{GeoSSPCA,SSPCA,PCA}

\end{axis}
\end{tikzpicture}
    \caption{Out-of-sample face recognition precision by method}
    \label{fig7}
\end{figure}
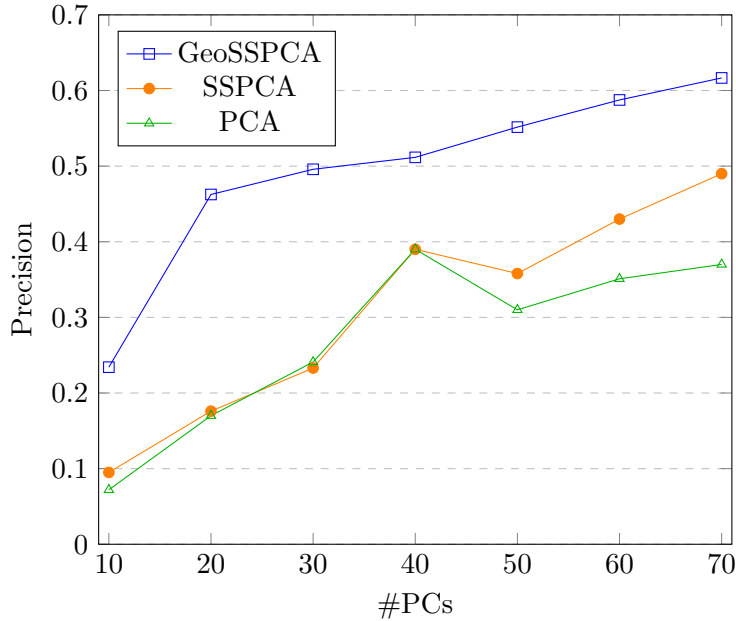
\medbreak
 We note that the patterns can be written also as intersections of half spaces by modifying accordingly (12). However, the polyhedron resulting from the relaxation of the binary constraints leads to computational considerations that go beyond the scope of this paper. 
\subsection{Discussion}
We deduct from the geometric interpretation of GeoSPCA some implications on good practices when using Sparse PCA in terms of choice of $a$ and $k$. First on the choice of $a$, if the matrix $\mathbf{X}$ can be approximated with a matrix $\mathbf{X}'$ of rank $a^*$, it seems unlikely that a submatrix $\mathbf{X}(\mathbf{s})$ of $\mathbf{X}$ would need to be approximated by a matrix of rank higher than $a^*$. Choosing a value $a$ higher than $a^*$ would lead to capture elements that were meant to be ignored such as noise and could then lead to over-fitting. It is however possible to consider values of $a$ that are lower than $a^*$. If $\mathbf{X}$ can be approximated by a matrix of rank $a^*$, a submatrix $\mathbf{X}(\mathbf{s})$ could eventually be approximated by a matrix $\mathbf{X}(\mathbf{s})'$ of rank $b<a$ if $\mathbf{X}(\mathbf{s})'$ has columns that are orthogonal to other columns of $\mathbf{X}'$, or in other words if the variables defined by the support of $\boldsymbol s$ are independent from other features of $\mathbf{X}$. This is true when the supports of the PCs are disjoint for example as we have considered in the Structured Sparse PCA setting in the current section. \medbreak
Regarding the choice of $k$, we notice first that choosing $k\leq a$ leads to a trivial problem as an optimal solution can be constructed using the $k$ columns of $\mathbf{X}$ that have the largest norm. In this case, $k$ columns can be projected into a space of dimension $a$ using $\mathbf{U}[\mathbf{s}]$ with $\mathbf{X}(\mathbf{s}) = \mathbf{U}[\mathbf{s}]\mathbf{U}[\mathbf{s}]^T\mathbf{X}(\mathbf{s})$. On the other hand, $k$ cannot be chosen too big to preserve the putpose of sparse PCA. We provide a summary in Figure \ref{fig8}.

\begin{figure}
\centering
    \includegraphics[ scale=0.45]{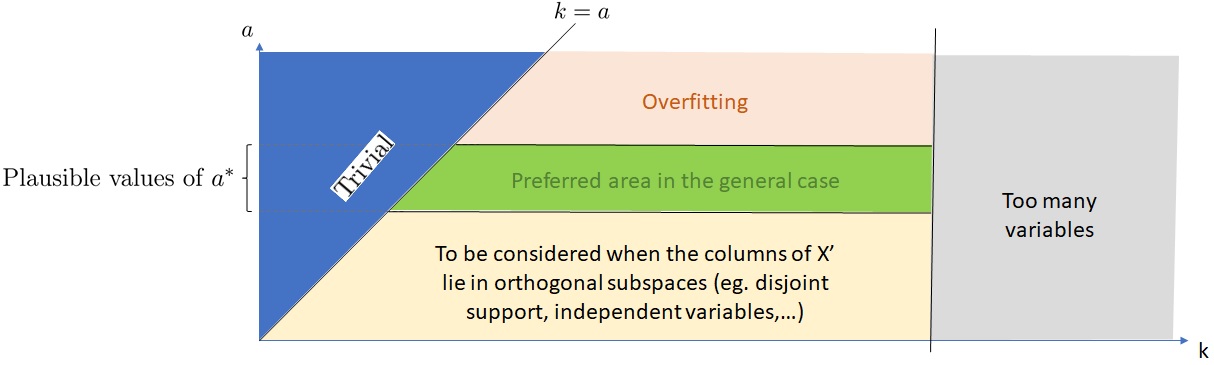}
    \caption{Map for the choice of $a$ and $k$.}
    \label{fig8}
\end{figure} 
\section{Conclusion}
In this paper, we proposed GeoSPCA, a new approach to the sparse PCA problem building on a geometrical interpretation of the problem. We addressed in particular the case in which the PCs share a common support. We then illustrated the versatility of this method to the case in which PCs are organized in groups that have disjoint supports and further extended this adaptation to the Structured Sparse PCA problem. The experiments we conducted showed that GeoSPCA can tackle real world instances with a number of features in the 10,000s exceeding the performance of state-of-the-art approaches while providing high quality solutions in minutes.\medbreak
We believe the method can be further applied to more variants of the Sparse PCA problem and can also be improved especially by generating more efficient cuts at each iteration of the algorithms proposed.

\bibliography{sample}

\end{document}